\documentclass[12pt]{amsart}
\usepackage[margin=1in,includeheadfoot]{geometry}
\usepackage{amsmath}
\usepackage{slashed}
\usepackage{graphicx}

\usepackage{mathrsfs}
\usepackage{txfonts}
\usepackage{amssymb,dsfont}
\usepackage{enumerate}
\usepackage{slashed}
\usepackage{fancyhdr}
\usepackage{aliascnt}
\usepackage{pdfsync}
\usepackage{hyperref}

\newtheorem{thm}{Theorem}[section]
\newtheorem{cor}[thm]{Corollary}
\newtheorem{lem}[thm]{Lemma}
\newtheorem{prop}[thm]{Proposition}
\theoremstyle{definition}
\newtheorem{defn}[thm]{Definition}
\theoremstyle{remark}
\newtheorem{rem}[thm]{Remark}
\numberwithin{equation}{section}

\newcounter{stepnum}

\def\bee{\begin{eqnarray}}
\def\beee{\begin{eqnarray*}}
\def\eee{\end{eqnarray}}
\def\eeee{\end{eqnarray*}}

\def\ba{\begin{array}}
\def\ea{\end{array}}

\def\R{\mathbb R}

\begin{document}

\title[Asymptotic analysis for approximate harmonic maps from degenerating cylinders]{Asymptotic analysis for approximate harmonic maps from degenerating cylinders and applications to minimal surfaces}

\author[Li]{Jiayu Li}%
\address{School of Mathematical Sciences, University of Science and Technology of China, Hefei, 230026, People's Republic of China}%
\email{jiayuli@ustc.edu.cn}%

\author[Liu]{Lei Liu}
\address{School of Mathematics and Statistics, Key Laboratory of Nonlinear Analysis and Applications (Ministry of Education), Hubei Key Laboratory of Mathematical Sciences, Central China Normal
University, Wuhan, 430079, People's Republic of China}%
\email{leiliu2020@ccnu.edu.cn}
\author[Zhu]{Miaomiao Zhu}
\address{School of Mathematical Sciences \\ Shanghai Jiao Tong University \\ Dongchuan Road 800 \\ Shanghai, 200240, People's Republic of China}
\email{mizhu@sjtu.edu.cn}

\thanks{April 27, 2026}


\date{}
\begin{abstract}
We investigate the blow-up analysis and quantitative behavior for a sequence of maps $\{u_n\}_{n=1}^\infty$ from degenerating tori $(T^2,g_n)$ or from degenerating cylinders $(S^1\times [0,\pi],g_n)$ with free boundary conditions $u_n(S^1\times \{0,\pi\})\subset K$ to a compact Riemannian manifold $(N,h)$ satisfying $$E(u_n)+\|\tau(u_n,g_n)\|_{L^2}\leq \Lambda<\infty,$$ where $\tau(u_n,g_n)$ is the tension field of $u_n$, $K\subset N$ is a smooth submanifold. We establish generalized energy identities and prove that away from bubbles, the asymptotic limit of the necks are either some geodesics on $N$ or some geodesic-like curves on $K$ where some length formulas are given. This partially confirms a conjecture by Ding-Li-Liu \cite{Ding-Li-Liu} in the sense of approximate sequence case.

Moreover, we study an evolution system to seek minimal cylinders in a compact Riemannian manifold with free boundary and with arbitrary codimensions. By studying the convergence of the flow at infinity, we obtain some existence results of minimal cylinders with free boundary. Compared with the closed case in, an interesting new phenomenon here is that the neck may converges to a geodesic-like curve on $K$.
\end{abstract}
\maketitle

\section{Introduction}

\

Let $(M,g)$ be a compact Riemannian manifold with smooth boundary $\partial M$ and $(N,h)$ a compact Riemannian manifold of dimension $n$. For a map $u\in C^2(M,N)$, the energy density of $u$ is defined by
\[
e(u)=\frac{1}{2}|\nabla_g u|^2=\frac{1}{2}{\rm Trace}_g \ u^*h,
\]
where $u^*h$ is the pull-back of the metric tensor $h$. The energy of the map $u$ is defined as $$E(u)=\int_Me(u)dvol_g.$$

Let $K\subset N$ be a $k$-dimensional submanifold where $1\leq k\leq n-1$. Define the class $$C(K)=\{u\in C^2(M,N);u(\partial M)\subset K\}.$$ A critical point of the energy functional $E(\cdot,\cdot)$ over $C(K)$ is called a harmonic map from $(M,g)$ into $(N,h)$ with free boundary $u(\partial M)\subset K$. By Nash's embedding theorem,  $(N,h)$ can be isometrically embedded into some Euclidean space $\mathbb{R}^N$. Then we derive the Euler-Lagrange equations
\[
-\Delta_g u=A(u)(\nabla_g u,\nabla_g u),
\]
where $A(\cdot,\cdot)$ is the second fundamental form of $N\subset \mathbb{R}^N$ and $\Delta_g$ is the Laplace-Beltrami operator on $(M,g)$ which is defined by $$\Delta_g:=\frac{1}{\sqrt{g}}\frac{\partial}{\partial x^\beta} \left(\sqrt{g}g^{\alpha\beta}\frac{\partial}{\partial x^\alpha} \right).$$ Moreover, for $1\leq k\leq n-1$, $u$ has free boundary $u(\partial M)$ on $K$, that is $$u(x)\in K,\ \ du(x)(\overrightarrow{n})\perp T_{u(x)}K,\ \ x\in \partial M,$$ where $\overrightarrow{n}$ is the outward unite normal vector on $\partial M$ and $\perp$ means orthogonal.

In dimension two, i.e. $(M,g)$ is a surface, harmonic maps have important applications to minimal surface. It is well known that a conformal harmonic map is in fact a minimal surface. Thus it provides a way to obtain minimal surfaces (with arbitrary codimension) in a give Riemannian manifold by finding conformal harmonic maps.

From the perspective of variational methods, this was achieved by Schoen-Yau \cite{Schoen-Yau}, Sacks-Uhlenbeck \cite{SU}, Sacks-Uhlenbeck \cite{SU-2} for closed surfaces. The idea is that for each conformal structure on the domain surface, one produces an energy minimizing map which is a harmonic map, and then minimize over all conformal structures to produce a branched minimal surface. Later, this idea was used to produce minimal surfaces with free boundary, see Fraser \cite{Fraser} for the case of minimal disks  and Chen-Fraser-Pang \cite{Chen-Fraser-Pang} for high genus cases.

In \cite{Ding-Li-Liu}, Ding-Li-Liu introduced a heat flow approach to seek minimal surfaces in Riemannian manifold with arbitrary codimension by considering the negative $L^2$ gradient flow of the energy functional $E(u,g)$ given by
\begin{align}
\begin{cases}
\frac{\partial u}{\partial t}=\tau(u,g),\ \ &in\ \ M\times [0,T),\\
\frac{\partial g}{\partial t}=u^*h-e(u,g)g,\ \ &in\ \ M\times [0,T),
\end{cases}
\end{align}
where $\tau(u,g):=\Delta_gu+A(u)(\nabla_gu,\nabla_gu)$ is the tension field of the map $u$. In order to have a good representation of the space of domain metrics, they considered the torus case, i.e. $M=T^2$. Later, the higher genus extension of this flow method was developed by Rupflin-Topping \cite{RT}, Rupflin \cite{R}, Rupflin-Topping-Zhu \cite{RTZ} etc. and the case of Douglas-Plateau problem for this flow was studied by Rupflin \cite{R2}.

Before stating our results, let us recall first the setting in \cite{Ding-Li-Liu}. Let $H$ be the upper half plane. For any $z=a+ib\in H$, we have a parallelogram $R_z$ determined by the vertices $0,1,z,z+1$. We can also view $R_z$ as a flat cylinder $\mathbf{C}/\{1,z\}$ with a conformal structure represented by the flat metric. Let $L_z$ be the unique linear map which maps the unit square $R_i$ onto $R_z$ with fixed point $1$. Denoting $ds^2$ the Euclidean metric on $\mathbf{C}$, then $$G_z=L_z^*ds^2$$ is a flat metric representing a conformal structure on $M=\mathbf{C}/\{1,i\}$. This gives a representation of the Teichm\"{u}ller  space by flat metrics. Furthermore, if we normalize the flat metric $G_z$ such that it has unit area, then all such metrics can be parameterized by two constant $(a,b)\in \mathbf{R}\times \mathbf{R}^+$, and given by following matrix
$$g_{a,b}=\frac{1}{b}G_{a+ib}=\frac{1}{b}\left(
                                                            \begin{array}{cc}
                                                              1 & a \\
                                                              a & a^2+b^2 \\
                                                            \end{array}
                                                          \right).
$$
This metric on $M$ corresponds to the conformal structure determined by the parallelogram $R_z$ with $z=a+ib$. Denote $\mathbf{g}$ the space of all such metrics.

For the metric $g=g_{a,b}\in\mathbf{g}$ on $M$, the energy is given as follows
\begin{align*}
E(u,g)=\frac{1}{2b}\int_M \left[(a^2+b^2)|\frac{\partial u}{\partial x^1}|^2+|\frac{\partial u}{\partial x^2}|^2-2a\frac{\partial u}{\partial x^1}\cdot \frac{\partial u}{\partial x^2}\right]dx^1dx^2.
\end{align*}

They \cite{Ding-Li-Liu} considered the following modified negative gradient flow that
\begin{align}\label{equ:13}
\begin{cases}
\frac{\partial u}{\partial t}=\tau(u,g),\ \ &in \ \ M\times [0,T),\\
\frac{da}{dt}=-b\int_M \left[a|\frac{\partial u}{\partial x^1}|^2-\frac{\partial u}{\partial x^1}\cdot \frac{\partial u}{\partial x^2}\right]dM,\ \ &in \ \ M\times [0,T),\\
\frac{db}{dt}=-\frac{1}{2}\int_M \left[(b^2-a^2)|\frac{\partial u}{\partial x^1}|^2-|\frac{\partial u}{\partial x^2}|^2+2a\frac{\partial u}{\partial x^1}\cdot \frac{\partial u}{\partial x^2}\right]dM,\ \ &in \ \ M\times [0,T),
\end{cases}
\end{align} with the initial data
\begin{equation}\label{equ:14}
u_{t=0}=u_0,\ \ g_{t=0}=g_0=g_{a_0,b_0},
\end{equation} where $g=g_{a,b}$. They proved that the metric $g=g_{a,b}$ along the flow \eqref{equ:13}-\eqref{equ:14} will not degenerate at any finite time. Moreover, they obtained following theorem that
\begin{thm}[Theorem 4.1 in \cite{Ding-Li-Liu}]\label{thm:03}
Suppose $(u,g)$ is a solution of \eqref{equ:13}-\eqref{equ:14} on $T^2\times (0,\infty)$ with initial data $(u_0,g_0)$. Then one of the following two cases will happen:
\begin{itemize}
\item[(i)] There exists $t_i\to\infty$ such that $g(t_i)\to g_\infty\in\mathbf{g}$ and $u(t_i)$ converges strongly to $u_\infty$ except at finitely points. There may exist a finite number of bubbles $\{w_i\}_{i=1}^m$ such that $$\lim_{t\to\infty}E(u(t),g(t))=E(u_\infty,g_\infty)+\sum_{i=1}^mE(w_i).$$ Moreover, $u_\infty$ has the same conjugacy class of $u_0$ and it is a branched minimal immersion unless it is a constant.

    \

\item[(ii)] There exist $t_i\to\infty$ such that $g(t_i)$ does not converge in $\mathbf{g}$ and there exist a finite number of bubbles $\{w_i\}_{i=1}^m$ such that $$\lim_{t\to\infty}E(u(t),g(t))=\sum_{i=1}^mE(w_i).$$ Moreover, if no bubble occurs at infinity, i.e. $\lim_{t\to\infty}E(u(t),g(t))=0$, then as $t\to\infty$, the image of $u(t)$ must subconverge to either a point or a closed geodesic in $N$.
\end{itemize}
\end{thm}

\

In \cite[Remark 4.2]{Ding-Li-Liu}, Ding-Li-Liu imposed the following conjecture corresponding to the second conclusion $\rm{(ii)}$ of above theorem:

\

\textbf{``We believe that in the second case, away from the bubbles, the limit of the image of $u(t)$ may also be either a point or a closed geodesic (at least a closed cycle of geodesic segments)."}

\

Our first result is to partially answer the above conjecture in the sense of approximate sequence case. Before stating the results, we need some notations. Let $g_n$ be a sequence of degenerated metrics on $T^2$ with volume $1$. Let $\rho_n$ be the length of the shortest closed geodesic on $(T^2,g_n)$ which goes to zero as $n\to\infty$. Then we can see $(T^2,g_n)$ as a cylinder
$$\left([-\frac{1}{2}\rho_n^{-1},\frac{1}{2}\rho_n^{-1}]\times \rho_n\frac{1}{2\pi}S^1, ds^2+d\theta^2\right).$$
For a sequence of maps $u_n:(T^2,g_n)\to (N,h)$ with tension fields $\tau(u_n,g_n)$, we define the following quantity that
\begin{equation}\label{equ:01}
\alpha_n:=\frac{\rho_n}{2\pi}\left(\int_{\{s\}\times S^1}\left (|\partial_s u_n|^2-|\partial_\theta u_n|^2\right )d\theta -2\int_{[0,s]\times S^1}\tau( u_n,g_n)\cdot \partial_s u_n dsd\theta\right),\quad   (s,\theta)\in (T^2,g_n).\end{equation}
We will show that the quantity $\alpha_n$ is independent of $s$, i.e. $\alpha_n$ is a constant which can be considered as a Pohozaev type constant to measure the extent to which the Pohozaev type identity fails. See Lemma \ref{lem:pohozaev-constant}. Denote the Hopf differential form of $u_n$ by $\phi(u_n)dz^2$ where $z=s+\sqrt{-1}\theta$ and $$\phi(u_n):=\left|\frac{\partial u_n}{\partial s}\right|^2-\left|\frac{\partial u_n}{\partial \theta}\right|^2-2\sqrt{-1}\frac{\partial u_n}{\partial s}\cdot \frac{\partial u_n}{\partial \theta}.$$

Our first main result is the following
\begin{thm}\label{thm-main:01}
Let $\{u_n\}_{n=1}^\infty$ be a sequence of maps from degenerating tori $(T^2,g_n)$ to $(N,h)$ with tension field $\tau(u_n,g_n)$ such that $$E(u_n)\leq\Lambda<\infty,\ \ \lim_{n\to\infty}\rho_n\cdot \|\tau(u_n,g_n)\|_{L^2(T^2)}=0, \ \  \int_{T^2}dvol_{g_n}\equiv 1.$$  Then there exist finitely many bubbles $w^k:S^2\to N,\ k=1,...,m$, which are harmonic spheres, such that the following conclusions hold:
\begin{itemize}
\item[(1)] if $\lim_{n\to\infty}\frac{\|\tau(u_n,g_n)\cdot \frac{\partial u_n}{\partial s}\|_{L^1(T^2)}}{\rho_n}=0,$ then the following generalized energy identity holds
    $$\lim_{n\to\infty}E(u_n)=\sum_{i=1}^mE(w_i)+\lim_{n\to\infty}\frac{2\pi \alpha_n}{\rho_n^2}.$$

    \

   \item[(2)] if $\lim_{n\to\infty}\|\ |\phi(u_n)| \ \|_{L^1(T^2)}=0,$ then the following energy identity holds
   $$\lim_{n\to\infty}E(u_n)=\sum_{i=1}^mE(w_i).$$

\

\item[(3)] if
\begin{equation}\label{equ:17}
\lim_{n\to\infty}\left(\frac{\|\tau(u_n,g_n)\|_{L^2(T^2)}}{\rho_n}+\frac{\|\tau(u_n,g_n)\cdot \frac{\partial u_n}{\partial s}\|_{L^1(T^2)}}{\rho^3_n}\right)=0,
\end{equation} then away from the bubbles $\{w^i\}_{i=1}^m$, the asymptotic limit of $u_n$ are some geodesics on $N$. Moreover, the sum of the lengths of these geodesics is $$\lim_{n\to\infty}\sqrt{|\alpha_n|}\frac{\sqrt{2\pi^3}}{\rho^2_n}.$$
\end{itemize}
\end{thm}

\

\begin{rem}
In Theorem \ref{thm-main:01}, the length of geodesic can be finite or infinite and a geodesic of zero length means that all the bubbles are connected in the target manifold.
\end{rem}

\begin{rem}
The assumptions of Theorem \ref{thm-main:01} is more general and weaker than Theorem \ref{thm:03} at least by the following two aspects:
\begin{itemize}
  \item[(1)] Compared to the flow case where we can use an additional structure that the map will become more and more conformal along the flow, i.e. $\lim_{n\to\infty}\|\ |\phi(u_n)| \ \|_{L^1(T^2)}=0,$ the first conclusion can be applied for more generalized sequences.
  \item[(2)] If there is no bubble occurs at infinity time, then as in \cite{Ding-Li-Liu}, we can chose a time sequence going to infinity such that the condition \eqref{equ:17} automatically holds.
\end{itemize}
\end{rem}

\begin{rem}
When $\tau(u_n,g_n)\equiv 0$, i.e. $u_n$ is a sequence of harmonic maps, the results in  Theorem \ref{thm-main:01} were already proved in \cite{Zhu-3, Chen-Li-Wang}.
\end{rem}

As an application of Theorem \ref{thm-main:01}, we shall provide a new proof of Theorem \ref{thm:03} (or \cite[Theorem 4.1]{Ding-Li-Liu}). In fact, we shall improve the conclusions of Theorem \ref{thm:03} by giving a length formula.
\begin{thm}\label{thm-main:02}
Under the assumptions of Theorem \ref{thm:03}, if there exist $t_n\to\infty$ such that $g(t_n)$ does not converge in $\mathbf{g}$, then there exist a finite number of bubbles $\{w_i\}_{i=1}^m$ such that $$\lim_{n\to\infty}E(u(t_n),g(t_n))=\sum_{i=1}^mE(w_i).$$ Moreover, the following two alternatives hold:
\begin{itemize}
\item[(1)] if $\lim_{n\to\infty}\left(\frac{\|\tau(u(t_n),g(t_n))\|_{L^2(T^2)}}{\rho_n}+\frac{\|\tau(u(t_n),g(t_n))\cdot \frac{\partial u(t_n)}{\partial s}\|_{L^1(T^2)}}{\rho^3_n}\right)=0,$ then away from bubbles $\{w^i\}_{i=1}^m$, the asymptotic limit of $u_n$ are some geodesics on $N$. Moreover, the sum of the lengths of the geodesics is $$\lim_{n\to\infty}\sqrt{|\alpha_n|}\frac{\sqrt{2\pi^3}}{\rho^2_n}.$$

    \

\item[(2)]if no bubble occurs at infinity, i.e. $\lim_{t\to\infty}E(u(t),g(t))=0$, then as $t\to\infty$, the image of $u(t)$ must subconverge to either a point or a closed geodesic in $N$ with finite length $$\lim_{n\to\infty}\sqrt{|\alpha_n|}\frac{\sqrt{2\pi^3}}{\rho^2_n}.$$

\end{itemize}
\end{thm}

\

The second goal of this paper is to seek minimal cylinders in compact Riemannian manifold with free boundary. Roughly speaking, we shall extend Theorem \ref{thm:03} to the free boundary case.

Let $M=S^1\times [0,\pi]$ be a cylinder. Let $K\subset N$ be a compact smooth submanifold. One may find the negative $L^2$ gradient flow of $E(u,g)$ over the class $C(K)$ given by
\begin{align}
\begin{cases}
\frac{\partial u}{\partial t}=\tau(u,g),\ \ &in\ \ M\times [0,T),\\
\frac{\partial g}{\partial t}=u^*h-e(u,g)g,\ \ &in\ \ M\times [0,T),\\
u(x,t)\in K,\ \ \frac{\partial u}{\partial \overrightarrow{n}}\bot T_uK,\ \ &on\ \ \partial M\times [0,T),
\end{cases}
\end{align} where $\overrightarrow{n}$ is the unit outer normal vector.  Inspired by the closed case in \cite{Ding-Li-Liu}, we consider the following modified evolution system that
\begin{align}\label{equ:15}
\begin{cases}
\frac{\partial u}{\partial t}=\tau(u,g),\ \ &in \ \ M\times [0,T),\\
\frac{da}{dt}=-b\int_M \left[a|\frac{\partial u}{\partial x^1}|^2-\frac{\partial u}{\partial x^1}\cdot \frac{\partial u}{\partial x^2}\right]dM,\ \ &in \ \ M\times [0,T),\\
\frac{db}{dt}=-\frac{1}{2}\int_M \left[(b^2-a^2)|\frac{\partial u}{\partial x^1}|^2-|\frac{\partial u}{\partial x^2}|^2+2a\frac{\partial u}{\partial x^1}\cdot \frac{\partial u}{\partial x^2}\right]dM,\ \ &in \ \ M\times [0,T),\\
u(x,t)\in K,\ \ \frac{\partial u}{\partial \overrightarrow{n}}\bot T_uK,\ \ &on\ \ \partial M\times [0,T),
\end{cases}
\end{align} with the initial data
\begin{equation}\label{equ:16}
u_{t=0}=u_0,\ \ g_{t=0}=g_0=g_{a_0,b_0},
\end{equation} where $g=g_{a,b}\in\mathbf{g}$.

\

In order to study the convergence of flow \eqref{equ:15}-\eqref{equ:16} near infinity time, we need to study the asymptotic properties near the free boundary. Let $(M,g_n)$ be a sequence of degenerated cylinders with volume $1$. Compared to the interior case, as discussed in \cite{Liu-Song-Zhu}, we need to consider two types of boundary degeneration which is more complicated.

\

\noindent{\bf Type $\mathbf{I}$ boundary degeneration}: collapsing a simple \textbf{closed} geodesic $\gamma_n$ with length $\rho_n$. In this case, we can also see $(M,g_n)$ as a cylinder $\left([-\frac{1}{2}\rho_n^{-1},\frac{1}{2}\rho_n^{-1}]\times \rho_n\frac{1}{2\pi}S^1, ds^2+d\theta^2\right)$. For a sequence of maps $u_n:(M,g_n)\to (N,h)$ with tension fields $\tau(u_n,g_n)$, we define a Pohozaev type constant $\alpha_n$ as in \eqref{equ:01} for interior case.

\

\noindent{\bf Type $\mathbf{II}$ boundary degeneration}: collapsing a simple geodesic $\gamma_n $ which connects two points on the boundary $S^1\times \{0,\pi\}$. In this case, we can see $(M,g_n)$ as  $\left([-\frac{1}{2}\rho_n^{-1},\frac{1}{2}\rho_n^{-1}]\times \rho_n\frac{1}{\pi}[0,\pi],ds^2+d\theta^2\right)$. For a sequence of maps $u_n:(M,g_n)\to (N,h)$ with tension fields $\tau(u_n,g_n)$ and with free boundary condition $u_n|_{\partial M}\subset K$, where the boundary $\partial M=[-\frac{1}{2}\rho_n^{-1},\frac{1}{2}\rho_n^{-1}]\times \{0,\pi\}$,  we can also define a Pohozaev type constant
\begin{equation}\label{equ:02}
\alpha_n:=\frac{\rho_n}{2\pi}\left(\int_{\{s\}\times [0,\pi]}\left (|\partial_s u_n|^2-|\partial_\theta u_n|^2\right )d\theta -2\int_{[0,s]\times [0,\pi]}\tau( u_n,g_n)\cdot \partial_s u_n dsd\theta\right),   (s,\theta)\in (M,g_n).\end{equation}
See Lemma \ref{lem:pohozaev-constant} for more details.

\

For a sequence of maps with uniformly bounded energy, when boundary degeneration occurs, we prove the following:
\begin{thm}\label{thm-main:03}
Let $\{u_n\}_{n=1}^\infty$ be a sequence of maps from $(S^1\times [0,\pi],g_n)$ to $(N,h)$ with tension field $\tau(u_n,g_n)$ and with free boundaries $u_n(S^1\times \{0,\pi\})\subset K$, such that $$E(u_n)\leq\Lambda,\ \ \lim_{n\to\infty}\rho_n\cdot \|\tau(u_n,g_n)\|_{L^2(M)}=0,\ \ \int_{S^1\times [0,\pi]}dvol_{g_n}\equiv 1,$$ where $\rho_n$ is the length of the shortest geodesic $\gamma_n$ on $(S^1\times [0,\pi],g_n)$ which goes to zero as $n\to\infty$. Then there exist finitely many bubbles $\{w^k\}_{ k=1}^m$, where each $w^k$ is either a harmonic map from $S^2$ to $N$ or a harmonic map from the unit Euclidean disk $D_1(0)$ to $N$ with free boundary $w^k(\partial D_1(0))\subset K$, such that  the following  conclusions hold:

\

\begin{itemize}
\item[(1)] {\bf Type $\mathbf{I}$  boundary degeneration}: collapsing a simple closed geodesic.

\

\begin{itemize}
\item[(1-1)] If $\lim_{n\to\infty}\frac{\|\tau(u_n,g_n)\cdot \frac{\partial u_n}{\partial s}\|_{L^1(M)}}{\rho_n}=0,$ then the following generalized energy identity holds
     $$\lim_{n\to\infty}E(u_n)=\sum_{i=1}^mE(w_i)+\lim_{n\to\infty}\frac{2\pi\alpha_n}{\rho_n^2},$$
     where $\alpha_n$ is the Pohozaev type constant defined by \eqref{equ:01}.

\

\item[(1-2)] If  $\lim_{n\to\infty}\|\ |\phi(u_n)| \ \|_{L^1(M)}=0,$ then the following energy identity holds
$$\lim_{n\to\infty} E(u_n)=\sum_{i=1}^mE(w_i).$$

\

\item[(1-3)] If $\lim_{n\to\infty}\left(\frac{\|\tau(u_n,g_n)\|_{L^2(M)}}{\rho_n}+\frac{\|\tau(u_n,g_n)\cdot \frac{\partial u_n}{\partial s}\|_{L^1(M)}}{\rho^3_n}\right)=0,$ then away from the bubbles $\{w^i\}_{i=1}^m$, the asymptotic limit of $u_n$ are some geodesics on $N$. Moreover, the sum of the lengths of these geodesics is
    $$\lim_{n\to\infty}\sqrt{|\alpha_n|}\frac{\sqrt{2\pi}}{\rho^2_n}.$$
\end{itemize}

\

\item[(2)] {\bf Type $\mathbf{II}$  boundary degeneration}: collapsing a simple geodesic $\gamma_n $ which connects two points on the boundary $S^1\times \{0,\pi\}$.

\

\begin{itemize}
\item[(2-1)] If $\lim_{n\to\infty}\frac{\|\tau(u_n,g_n)\cdot \frac{\partial u_n}{\partial s}\|_{L^1(M)}}{\rho_n}=0,$ then the following weak energy identity holds
     $$\lim_{n\to\infty}E(u_n)=\sum_{i=1}^mE(w_i)+\lim_{n\to\infty}\frac{\pi\alpha_n}{\rho_n^2},$$
     where $\alpha_n$ is the Pohozaev type  constant defined by \eqref{equ:02}.

\

    \item[(2-2)] If  $\lim_{n\to\infty}\|\ |\phi(u_n)| \ \|_{L^1(M)}=0,$ then the following energy identity holds
    $$\lim_{n\to\infty}E(u_n)=\sum_{i=1}^mE(w_i).$$

\

\item[(2-3)] If $\lim_{n\to\infty}\left(\frac{\|\tau(u_n,g_n)\|_{L^2(M)}}{\rho_n}+\frac{\|\tau(u_n,g_n)\cdot \frac{\partial u_n}{\partial s}\|_{L^1(M)}}{\rho^3_n}\right)=0,$ then away from bubbles $\{w^i\}_{i=1}^m$, the asymptotic limit of $u_n$ are some geodesic-like curves on $K$. Moreover, the sum of the lengths of the geodesic-like curves is
    $$\lim_{n\to\infty}\sqrt{|\alpha_n|}\frac{\sqrt{\pi}}{\rho^2_n}.$$
    Here, a curve $\gamma\subset K$ denoted by $u(s)$ is called \textbf{a geodesic-like curve} iff it satisfies the following equation
    \begin{equation}\label{def:geodesic-like-curve}
    \frac{d^2}{ds^2}u(s)=-A(u)(\frac{d}{ds} u,\frac{d}{ds} u) +\frac{1}{2}D^2\sigma(u)(\frac{d}{ds} u,\frac{d}{ds} u),\end{equation} where $s$ is the arc length parameter and $\sigma$ is the evolution map corresponding to $K$ defined by \eqref{def:01}.
\end{itemize}

\end{itemize}
\end{thm}

\

We want to remark that all conditions of Theorem \ref{thm-main:03} hold automatically for zero tension fields, i.e. $\tau(u_n,g_n)\equiv 0$. When $u_n$ is a harmonic map sequence, the generalized energy identity was established by \cite{Liu-Song-Zhu}. Here, as a special case of  Theorem \ref{thm-main:03}, we can further characterize the limit of degenerate necks by some curves in a geometric manner. Precisely, we have the following corollary.

\

\begin{cor}
Let $\{u_n\}_{n=1}^\infty$ be a sequence of harmonic maps from $(S^1\times [0,\pi],g_n)$ to $(N,h)$  with free boundaries $u_n(S^1\times \{0,\pi\})\subset K$, such that $$E(u_n)\leq\Lambda,\ \  \int_{S^1\times [0,\pi]}dvol_{g_n}\equiv 1,$$ where $\rho_n$ is the length of the shortest geodesic $\gamma_n$ on $(S^1\times [0,\pi],g_n)$ which goes to zero as $n\to\infty$. Then there exist finitely many bubbles $\{w^k\}_{ k=1}^m$, where each $w^k$ is either a harmonic map from $S^2$ to $N$ or a harmonic map from the unit Euclidean disk $D_1(0)$ to $N$ with free boundary $w^k(\partial D_1(0))\subset K$, such that  the following  conclusions hold:

\

\begin{itemize}
\item[(1)] {\bf Type $\mathbf{I}$  boundary degeneration}: the following generalized energy identity holds     $$\lim_{n\to\infty}E(u_n)=\sum_{i=1}^mE(w_i)+\lim_{n\to\infty}\frac{2\pi\alpha_n}{\rho_n^2},$$
     where $\alpha_n$ is the Pohozaev type constant defined by \eqref{equ:01}.
Away from the bubbles $\{w^i\}_{i=1}^m$, the asymptotic limit of $u_n$ are some geodesics on $N$. Moreover, the sum of the lengths of these geodesics is
    $\lim_{n\to\infty}\sqrt{|\alpha_n|}\frac{\sqrt{2\pi}}{\rho^2_n}.$

\

\item[(2)] {\bf Type $\mathbf{II}$  boundary degeneration}: the following weak energy identity holds
     $$\lim_{n\to\infty}E(u_n)=\sum_{i=1}^mE(w_i)+\lim_{n\to\infty}\frac{\pi\alpha_n}{\rho_n^2},$$
     where $\alpha_n$ is the Pohozaev type  constant defined by \eqref{equ:02}.
Away from bubbles $\{w^i\}_{i=1}^m$, the asymptotic limit of $u_n$ are some geodesic-like curves on $K$. Moreover, the sum of the lengths of the geodesic-like curves is
    $\lim_{n\to\infty}\sqrt{|\alpha_n|}\frac{\sqrt{\pi}}{\rho^2_n}.$

\end{itemize}
\end{cor}

\

As an application of Theorem \ref{thm-main:03} to the flow system \eqref{equ:15}-\eqref{equ:16}, we have
\begin{thm}\label{thm-main:04}
Suppose $(u,g)$ is a solution of \eqref{equ:15}-\eqref{equ:16} on $(S^1\times [0,\pi])\times (0,\infty)$ with initial data $(u_0,g_0)$. Then one of the following two cases will happen:
\begin{itemize}
\item[(i)] There exists $t_i\to\infty$ such that $g(t_i)\to g_\infty\in\mathbf{g}$ and $u(t_i)$ converges strongly to $u_\infty$ except at finitely points. There may exist a finite number of bubbles $\{w_i\}_{i=1}^m$ where $w^i$ is either a harmonic map from $S^2$ to $N$ or a harmonic map from unit Euclidean disk $D_1(0)$ to $N$ with free boundary $w^k(\partial D_1(0))\subset K$,  such that $$\lim_{t\to\infty}E(u(t),g(t))=E(u_\infty,g_\infty)+\sum_{i=1}^mE(w_i).$$ Moreover, $u_\infty$ has the same conjugacy class of $u_0$ and it is a branched minimal cylinder with free boundaries $u_\infty(S^1\times \{0,\pi\})\subset K$ unless it is a constant.

    \

\item[(ii)] There exist $t_n\to\infty$ such that $g(t_n)$ does not converge in $\mathbf{g}$, i.e. there exists a geodesic $\gamma_n\subset (S^1\times [0,\pi],g(t_n))$ with length $\rho_n$ which goes to zero as $n\to\infty$.  Then there exist a finite number of bubbles $\{w_i\}_{i=1}^m$ where $w^i$ is either a harmonic map from $S^2$ to $N$ or a harmonic map from unit Euclidean disk $D_1(0)$ to $N$ with free boundary $w^k(\partial D_1(0))\subset K$,  such that $$\lim_{t\to\infty}E(u(t),g(t))=\sum_{i=1}^mE(w_i).$$

    Moreover, the following two alternatives hold:
\begin{itemize}
\item[(ii-1)] if $\lim_{n\to\infty}\left(\frac{\|\tau(u(t_n),g(t_n))\|_{L^2(M)}}{\rho_n}+\frac{\|\tau(u(t_n),g(t_n))\cdot \frac{\partial u(t_n)}{\partial s}\|_{L^1(M)}}{\rho^3_n}\right)=0,$ then away from bubbles $\{w^i\}_{i=1}^m$, the asymptotic limit of $u_n$ are either some geodesics on $N$ with the sum of the lengths $\lim_{n\to\infty}\sqrt{|\alpha_n|}\frac{\sqrt{2\pi}}{\rho^2_n}$ or some geodesic-like curves on $N$ with the sum of the lengths of these curves  $\lim_{n\to\infty}\sqrt{|\alpha_n|}\frac{\sqrt{\pi}}{\rho^2_n}.$

\item[(ii-2)] if there is no bubble occurs at infinity time, i.e. $\lim_{t\to\infty}E(u(t),g(t))=0$, and $$\limsup_{t\to\infty}\rho(t)^p\cdot \|\tau(u(t),g(t))\|_{L^2(M)}\leq C ,\ \ \limsup_{t\to\infty}t^r\rho(t)\leq C$$ for some constants $p\in [0,1)$, $r\in (0,\frac{1}{100})$, where $\rho(t)$ is the length of the shortest geodesic on $(M,g(t))$, then as $t\to\infty$, the image of $u(t)$ must subconverge to either a point or a geodesic in $N$ with length $\lim_{n\to\infty}\sqrt{|\alpha_n|}\frac{\sqrt{2\pi}}{\rho^2_n}$ or a geodesic-like curve in $K$ with length $\lim_{n\to\infty}\sqrt{|\alpha_n|}\frac{\sqrt{\pi}}{\rho^2_n}.$
\end{itemize}
\end{itemize}
\end{thm}

\

\begin{rem}
As a corollary of Theorem \ref{thm-main:04}, we can get a similar result of Chen-Fraser-Pang \cite{Chen-Fraser-Pang} for the free boundary case by imposing some incompressibility assumption. Let $M=S^1\times [0,\pi]$ be a cylinder and $K\subset N$ be a compact smooth submanifold. If $f\in C^1(M,N)$ satisfies free boundary $f(\partial M)\subset K$, which induces an injection $f_\ast :\pi_1(M)\times \pi_1(M,\partial M)\to \pi_1(N)\times \pi_1(N,K)$, then there is a branched minimal immersions $u:M\to N$ with free boundary $u(\partial M)\subset K$ which induces the same injection between the fundamental groups.
\end{rem}

\

To prove Theorem \ref{thm-main:01} and Theorem \ref{thm-main:03}, comparing to \cite{Zhu-3, Chen-Li-Wang} where they studied asymptotic behavior for a sequence of harmonic maps from degenerated surfaces, we need to overcome two additional difficulties. The first one is that the sequence is not harmonic maps, but with approximate $L^2$-bounded tension field terms. On one hand, recall that, to measure the extent to which the Pohozaev type identity fails, by using the property that the Hopf differential is holomorphic since the map is harmonic, a Pohozaev type constant is introduced in \cite{Zhu-3} and further explored in \cite{Chen-Li-Wang}. This quantity plays an important role to characterize generalized energy identities and limits of necks.  Here, by integrating by parts, we can still define a Pohozaev type constant   $\alpha_n$ (see Definition \ref{def:Pohozaev-cons}), where one can see that besides Pohozaev type constant, there is an additional error term involving the tension field. We need to compare the speed of $|\alpha_n|$ with  this error term. The second difficulty is to deal with boundary degeneration, where the way of degeneration  is more complicated than the interior case. It is an interesting and new phenomenon that the degenerating boundary neck may converge to a geodesic-like curve which is very different from interior case such as \cite{Chen-Tian,Ding-Li-Liu}.

As a direct application of Theorem \ref{thm-main:03}, we want to use it to study the convergence of the flow \eqref{equ:15}-\eqref{equ:16}. Precisely, we need to find a time sequence $t_n$ which goes to infinity, such that $u(t_n,x)$ and $\tau(u_n,g_n):=\partial_tu(t_n,x)$ satisfy the conditions of   Theorem \ref{thm-main:03}. For interior case in \cite{Ding-Li-Liu}, if there are no bubble occurs at infinity, their key idea is to show the tangential energy is a high order term corresponding to the energy (or radial energy), i.e. $$\int |\frac{\partial u}{\partial \theta}|^2dx\leq CE^q,\ \ q>1.$$  Here, there appears a bad boundary integration term $$\int_{\{t\}\times S^1}|\nabla u|^2d\theta$$  which is just homogenous to $E$, i.e. $q=1$. See Lemma \ref{lem:06}. To deal with this  boundary integration term, we first extend the map across the free boundary and then use a decay estimate obtained in the proof of  Theorem \ref{thm-main:03} to improve the  tangential energy estimate a little bit in the sense that it can be controlled by $\rho(t)^2E$, where $\rho(t)$ characterizes the speed of degeneration. See Lemma \ref{lem:05}  and  Lemma \ref{lem:06}.

\

The rest of  the paper is organized as follows. In Section \ref{sec:energy-estimate}, we will study the asymptotic and quantitative behavior for a sequence of maps from standard cylinders and half cylinders with free boundaries.  We will established some generalized energy identities. In Section \ref{sec:refined-analysis}, by exploring a refined neck analysis, we will prove that away from bubbles, the limit of the maps is either a geodesic on $N$ or a geodesic-like curve on $K$. The proof of Theorem \ref{thm-main:01}, Theorem \ref{thm-main:02} and Theorem \ref{thm-main:03} will also be given in this part. In Section \ref{sec:flow}, we will study the evolution system of minimal cylinders with free boundary and prove Theorem \ref{thm-main:04}.

\

\section{Energy estimates on cylinders and half cylinders}\label{sec:energy-estimate}

\

In this section, we will study energy estimates for a sequence of maps with uniformly bounded energy on cylinders or half cylinders. Precisely, we shall establish some generalized energy identities in Theorem \ref{thm:01}.

Let $\lim_{n\to\infty}T_n=\infty$ and $u_n:P_n\to N$ be a sequence of maps from  cylinder $$P_n=[-T_n,T_n]\times S^1$$  satisfying $$E(u_n;P_n)\leq\Lambda,\ \ \lim_{n\to\infty}\|\tau (u_n)\|_{L^2(P_n)}=0$$ or defined on the vertical half cylinder $$Q_n:=[-T_n,T_n]\times [0,\pi]$$ with free boundary $u_n([-T_n,T_n]\times \{0,\pi\})$ on $K$, satisfying $$E(u_n;Q_n)\leq\Lambda,\ \ \lim_{n\to\infty}\|\tau (u_n)\|_{L^2(Q_n)}=0,$$ where $\tau(u_n)$ is the tension field of $u_n$. Here, the metric on $P_n$ and $Q_n$ is the standard Euclidean metric $dt^2+d\theta^2$.

\

By integrating by parts, we firstly define some constants as follows, which are important in our analysis to measure the extent to which the Pohozaev type identity fails.
\begin{lem}\label{lem:pohozaev-constant}
If $u$ is a map defined on $Q_n$ with with tension field $\tau(u)$ and with free boundaries $u([-T_n,T_n]\times \{0,\pi\})$ on $K$, then $$\frac{1}{2}\int_{\{t\}\times [0,\pi]}\left(|\partial_t u|^2-|\partial_\theta u|^2 \right)d\theta-\int_{[0,t]\times [0,\pi]}\tau( u)\cdot \partial_t ud\theta$$ is independent of $t$, where $\partial_t u=\frac{\partial u}{\partial t}$ and $\partial_\theta u=\frac{\partial u}{\partial \theta}$.

If $u$ is a map defined on $P_n$ with with tension field $\tau(u)$, then
$$\frac{1}{2}\int_{\{t\}\times S^1}\left(|\partial_t u|^2-|\partial_\theta u|^2 \right)d\theta-\int_{[0,t]\times S^1}\tau( u)\cdot \partial_t ud\theta$$ is independent of $t$.
\end{lem}

\begin{proof}
We shall only prove the first case, since the second case is similar and easier.

Since $u$ is a map defined on $Q_n$ with tension field $\tau(u)$, denoting $Q_{t_1t_2}=[t_1,t_2]\times [0,\pi]$, then we have
\begin{align*}
\int_{Q_{t_1t_2}}\tau( u)\cdot \partial_t ud\theta dt&=\int_{Q_{t_1t_2}}\Delta u\cdot \partial_t ud\theta dt\\
&=\int_{Q_{t_1t_2}}(\partial^2_t u+\partial^2_\theta u)\cdot \partial_t ud\theta dt\\
&=\frac{1}{2}\int_{Q_{t_1t_2}} \frac{\partial}{\partial t}(|\partial_tu|^2-|\partial_\theta u|^2) d\theta dt+\int_{t_1}^{t_2}\partial_t u\cdot\partial_\theta u|_0^{\pi}dt\\
&=\frac{1}{2}\left(\int_{\{t_2\}\times [0,\pi]} (|\partial_tu|^2-|\partial_\theta u|^2) d\theta- \int_{\{t_1\}\times [0,\pi]} (|\partial_tu|^2-|\partial_\theta u|^2) d\theta  \right),
\end{align*}
where we used the free boundary condition, i.e. $$\partial_t u\cdot \partial_\theta u|_{[-T_n, T_n]\times \{0,\pi\}}=0.$$ Then the conclusions of lemma follow immediately.
\end{proof}

\

By above lemma, we shall give the following definition of Pohozaev type constant.
\begin{defn}\label{def:Pohozaev-cons}
The constant
\begin{equation}
\alpha_n:=\int_{\{t\}\times [0,\pi]}\left (|\partial_t u_n|^2-|\partial_\theta u_n|^2\right)d\theta -2\int_{[0,t]\times [0,\pi]}\tau( u_n)\cdot \partial_t u_ndtd\theta\quad  \text{~for~}  Q_n\end{equation} or
\begin{equation}
\alpha_n:=\int_{\{t\}\times S^1}\left (|\partial_t u_n|^2-|\partial_\theta u_n|^2\right )d\theta -2\int_{[0,t]\times S^1}\tau( u_n)\cdot \partial_t u_n dtd\theta\quad  \text{~for~}  P_n\end{equation} is called the Pohozaev type constant for $u_n$ on $Q_n$ or $P_n$, respectively.
\end{defn}

\

We first consider a simpler case by assuming that there is no energy concentration points in $P_n$ and $Q_n$, i.e.
\begin{equation}\label{equ:04}
\lim_{n\to\infty}\sup_{-T_n\leq t\leq T_n-1}E(u_n, [t,t+1]\times S^1)=0.
\end{equation}
and
\begin{equation}\label{equ:03}
\lim_{n\to\infty}\sup_{-T_n\leq t\leq T_n-1}E(u_n, [t,t+1]\times [0,\pi])=0.
\end{equation}

\

Recall the notation that $$\phi(u_n):=\left|\frac{\partial u_n}{\partial t}\right|^2-\left|\frac{\partial u_n}{\partial \theta}\right|^2-2\sqrt{-1}\frac{\partial u_n}{\partial t}\cdot \frac{\partial u_n}{\partial \theta}.$$
In this section, we will establish following theorem about some weak energy identities.
\begin{thm}\label{thm:01}
Let $N$ be a compact Riemannian manifold and $K\subset N$ is a smooth submanifold. Let  $u_n:P_n\to N$ be a sequence of maps from  cylinder $P_n=[-T_n,T_n]\times S^1$ with free boundary $u_n(\{-T_n,T_n\}\times S^1)$ on $K$ satisfying $$ E(u_n;P_n)\leq\Lambda,\ \ \lim_{n\to\infty}\|\tau (u_n)\|_{L^2(P_n)}=0$$ or defined on the vertical half cylinder $Q_n:=[-T_n,T_n]\times [0,\pi]$ with free boundary $u_n([-T_n,T_n]\times \{0,\pi\})$ on $K$ and satisfying $$ E(u_n;Q_n)\leq\Lambda,\ \ \lim_{n\to\infty}\|\tau (u_n)\|_{L^2(Q_n)}=0.$$

Suppose there is no energy concentration for $u_n$, i.e. \eqref{equ:04} or \eqref{equ:03} holds  respectively. The following conclusions hold.
\begin{itemize}
\item[(1)] If  $ \lim_{n\to\infty}\|\tau (u_n)\cdot \partial_t u_n\|_{L^1(P_n)}\cdot T_n=0,\ \ \lim_{n\to\infty}\|\tau (u_n)\cdot \partial_t u_n\|_{L^1(Q_n)}\cdot T_n=0,$ then we have
$$\lim_{n\to\infty} E(u_n;P_n)=\lim_{n\to\infty}4\alpha_n T_n$$ or $$\lim_{n\to\infty} E(u_n;Q_n)=\lim_{n\to\infty}2\alpha_n T_n,$$ where $\alpha_n$ is the corresponding Pohozaev type constant.

\

\item[(2)] If $ \lim_{n\to\infty}\|\ |\phi (u_n)|\ \|_{L^1(P_n)}=0,\ \ \lim_{n\to\infty}\| \ |\phi (u_n)|\ \|_{L^1(Q_n)}=0,$ then we have
$$\lim_{n\to\infty} E(u_n;P_n)=0$$ or $$\lim_{n\to\infty} E(u_n;Q_n)=0.$$
\end{itemize}
\end{thm}

\

Before proving the above theorem, we first recall the small energy regularity of approximate harmonic maps for the interior case \cite{SU,DingWeiyueandTiangang} and for the free boundary case \cite{jost-Liu-Zhu}.

\begin{lem}\label{lem:small-energy-regularity-1}
Let $u_n:P_n\to N$ be a sequence of approximate harmonic maps with free boundary $u_n(\{-T_n,T_n\}\times S^1)$ on $K$. Assuming \eqref{equ:04} holds, then we have
\begin{align*}
&\|u_n\|_{Osc([t-\frac{1}{2},t+\frac{1}{2}]\times S^1)}+\|\nabla u_n\|_{W^{1,2}([t-\frac{1}{2},t+\frac{1}{2}]\times S^1)}\\
&\leq C(N)\left(\|\nabla u_n\|_{L^2([t-1,t+1]\times S^1)}+ \|\tau( u_n)\|_{L^2([t-1,t+1]\times S^1)}\right),\quad  \forall\ -T_n+1\leq t\leq T_n-1
\end{align*}
and
\begin{align*}
&\|u_n\|_{Osc([-T_n,-T_n+1]\times S^1)}+\|\nabla u_n\|_{W^{1,2}([-T_n,-T_n+1]\times S^1)}\\&\leq C(K,N)\left(\|\nabla u_n\|_{L^2([-T_n,-T_n+2]\times S^1)}+ \|\tau( u_n)\|_{L^2([-T_n,-T_n+2]\times S^1)}\right),\end{align*}
\begin{align*}
\|u_n\|_{Osc([T_n-1,T_n]\times S^1)}+\|\nabla u_n\|_{W^{1,2}([T_n-1,T_n]\times S^1)}&\leq C(K,N)\left(\|\nabla u_n\|_{L^2([T_n-2,T_n]\times S^1)}+ \|\tau( u_n)\|_{L^2([T_n-2,T_n]\times S^1)}\right)
\end{align*}
when $n$ is big enough, where $$\|u_n\|_{Osc(\Omega)}:=\sup_{x,y\in \Omega}|u_n(x)-u_n(y)|.$$

\end{lem}

\
\begin{lem}\label{lem:small-energy-regularity}
Let $u_n:Q_n\to N$ be a sequence of approximate harmonic maps with free boundary $u_n([-T_n,T_n]\times \{0,\pi\})$ on $K$. Assuming \eqref{equ:03} holds, then  we have
\begin{align*}
&\|u_n\|_{Osc([t-\frac{1}{2},t+\frac{1}{2}]\times [0,\pi])}+\|\nabla u_n\|_{W^{1,2}([t-\frac{1}{2},t+\frac{1}{2}]\times [0,\pi])}
\\&\leq C(K,N)\left(\|\nabla u_n\|_{L^2([t-1,t+1]\times [0,\pi])}+ \|\tau( u_n)\|_{L^2([t-1,t+1]\times [0,\pi])}\right),\quad  \ -T_n+1\leq t\leq T_n-1.
\end{align*}
when $n$ is big enough.

In particular, when $n$ is big enough, for any positive constant $\epsilon>0$, the free boundary condition implies that the image of $u_n$ is contained in a small tubular neighborhood of $K$ in $N$, i.e. $$u_n([-T_n+1,T_n-1]\times [0,\pi])\subset K_{C(K,N)\epsilon},$$ where $K_{C(K,N)\epsilon}$ denotes the $C(K,N)\epsilon$-tubular neighborhood of $K$ in $N$.
\end{lem}

\

Denote by $K_{\delta_0}$ the $\delta_0$-tubular neighborhood of $K$ in $N$. Taking $\delta_0>0$ small enough, then for any $y\in K_{\delta_0}$,
there exists a unique projection $y'\in K$. Set $\overline{y}=exp_{y'}\{-exp^{-1}_{y'}y\}$. So we may define an involution $\sigma$, $i.e.$ $\sigma^2=Id$ as in
\cite{Hamilton, GJ, Scheven} by
\begin{equation}\label{def:01}
\sigma(y)=\overline{y} \quad for \quad y\in K_{\delta_0}.
\end{equation}
Then it is easy to check that the linear operator $D\sigma:TN|_{K_{\delta_0}}\to TN|_{K_{\delta_0}}$ satisfies $D\sigma(V)=V$ for $V\in TK$ and $D\sigma(\xi)=-\xi$ for $\xi\in T^\perp K$.

By Lemma \ref{lem:small-energy-regularity}, when $n$ is big enough, we can define an extension of $u_n$ to $\hat{Q}_n:=[-T_n+1,T_n-1]\times S^1$ that
\begin{align}\label{def:function}
\hat{u}_n(t,\theta)=
\begin{cases}
u_n(t,\theta),\quad &if \quad (t,\theta)\in [-T_n+1,T_n-1]\times [0,\pi];\\
\sigma(u_n(\rho(t,\theta))) ,\quad &if \quad (t,\theta)\in [-T_n+1,T_n-1]\times [\pi,2\pi],
\end{cases}
\end{align}
where $\rho(t,\theta):=(t,2\pi-\theta)$.

\

Now, we derive the equation for the extended map $\hat{u}_n$.
\begin{prop}\label{prop:01}
Let $u\in W^{2,p}([T_1,T_2]\times [0,\pi],N)$, $1\leq p\leq \infty$, be a map with free boundary $u([T_1,T_2]\times \{0,\pi\})$ on $K$. Let $u([T_1,T_2]\times [0,\pi])\subset K_{\delta_0}$, then the extended map $\hat{u}$ defined by \ref{def:function} satisfying $\hat{u}\in W^{2,p}([T_1,T_2]\times S^1)$ and
\begin{equation}\label{equ:11}
\Delta \hat{u}+\Upsilon_{\hat{u}}(\nabla\hat{u},\nabla\hat{u})=\hat{\Gamma}\quad in \quad [T_1,T_2]\times S^1,
\end{equation}
where $\Upsilon_{\hat{u}}(\cdot,\cdot)$ is a bounded bilinear form defined by
\begin{align*}
\Upsilon_{\hat{u}}(\cdot,\cdot)=
\begin{cases}
A(\hat{u})(\cdot,\cdot)\ &in\ [T_1,T_2]\times (0,\pi),\\
A(\hat{u})(\cdot,\cdot)-D^2\sigma|_{\sigma(\hat{u})}(D\sigma|_{\hat{u}}\circ\cdot\ ,
D\sigma|_{\hat{u}}\circ\cdot)\ &in\ [T_1,T_2]\times (\pi,2\pi);
\end{cases}\end{align*} satisfying
\[
|\Upsilon_{\hat{u}}(\nabla\hat{u},\nabla\hat{u})|\leq C(K,N)|\nabla\hat{u}|^2
\] and $\hat{\Gamma}$ is defined by
\begin{align*}
\hat{\Gamma}=
\begin{cases}
\tau(u)(x)\ &in\ [T_1,T_2]\times (0,\pi),\\
D\sigma|_{\sigma(\hat{u})}\circ \tau(u)(\rho(x))\ &in\ [T_1,T_2]\times (\pi,2\pi).
\end{cases}\end{align*}

\end{prop}
\begin{proof}
According to the properties of $D\sigma$, it is easy to see that $\hat{u}\in W^{2,p}([T_1,T_2]\times S^1)$ since $u$ satisfies free boundary condition.  Next, we derive the equation for $\hat{u}$.

Computing directly, we have  $$D\hat{u}(x)=D\Pi_N|_{\hat{u}(x)}\circ D\sigma|_{u(\rho(x))}\circ Du|_{\rho(x)}\circ D\rho|_{x},\ \ x=(t,\theta)\in (T_1,T_2)\times (\pi,2\pi)$$ which implies that
\begin{align*}
\Delta \hat{u}(x)=&D^2\Pi_N(\hat{u}(x))\big(D\hat{u}(x),D\hat{u}(x)\big)+D\Pi_N|_{\hat{u}(x)}\circ D^2\sigma|_{u(\rho(x))}( Du|_{\rho(x)}\circ D\rho|_{x}, Du|_{\rho(x)}\circ D\rho|_{x})\\
&+D\Pi_N|_{\hat{u}(x)}\circ D\sigma|_{u(\rho(x))}\circ \tau(u)(\rho(x))\\
=&-A(\hat{u})\big(D\hat{u},D\hat{u}\big)+D^2\sigma|_{\sigma(\hat{u})}( D\sigma \circ D\hat{u}, D\sigma \circ D\hat{u})+ D\sigma|_{\sigma(\hat{u})}\circ \tau(u)(\rho(x)),
\end{align*} where we used facts that $D\pi_N|_y:\R^N\to T_yN,\ y\in N$ is an orthogonal projection map, $$-D^2\Pi_N(y)=A(y):T_yN\otimes T_yN\to T^{\perp}_yN$$ is the second fundamental form.
\end{proof}

%

Now, we prove Theorem \ref{thm:01}.
\begin{proof}[\textbf{Proof of Theorem \ref{thm:01}:}]
We prove for the vertical half cylinder $Q_n$ case, since the cylinder $P_n$ case is similar and in fact easier.

Since $\lim_{n\to\infty}\|\tau(u_n)\|_{L^2(Q_n)}=0$ and \eqref{equ:03} holds, by Lemma \ref{lem:small-energy-regularity}, we know that for any $\epsilon>0$, there holds $$u_n([-T_n+1,T_n-1]\times [0,\pi])\subset K_{C(N)\epsilon},$$  when $n$ is big enough. Taking $\epsilon>0$ small such that $C(N)\epsilon\leq \delta_0$, then we can use the definition  to extend $u_n$ to $\widehat{u}_n$, which is defined on $[-T_n+1,T_n-1]\times S^1$ and satisfies equation \eqref{equ:11}.

Setting $$\hat{u}^*_n(t):=\frac{1}{2\pi}\int_0^{2\pi}\hat{u}_n(t,\theta)d\theta,\ t\in [-T_n+1,T_n-1],$$  by Lemma \ref{lem:small-energy-regularity}, we have
\begin{align}\label{inequ:09}
\|\hat{u}_n-\hat{u}_n^*\|_{L^\infty([-T_n+1,T_n-1]\times S^1)}&\leq \sup_{t\in [-T_n+1,T_n-1]}\|\hat{u}_n\|_{Osc(\{t\}\times S^1)}\notag\\
&\leq C(K,N)\sup_{t\in [-T_n+1,T_n-1]}\|u_n\|_{Osc(\{t\}\times [0,\pi])}\notag\\&\leq C(K,N)\epsilon,
\end{align}
when $n$ is big enough.

For any $t_n\in [-T_n+1,T_n-1]$, multiplying the equation \eqref{equ:11} by $\hat{u}_n-\hat{u}_n^*$ and integrating by parts, we get
\begin{align}\label{equ:12}
&\int_{[t_n+s,t_n-s]\times S^1}\Upsilon_{\hat{u}_n}(\nabla \hat{u}_n,\nabla \hat{u}_n)(\hat{u}_n-\hat{u}_n^*)dtd\theta-\int_{[t_n+s,t_n-s]\times S^1} \hat{\Gamma}(\hat{u}_n) (\hat{u}_n-\hat{u}_n^*) dtd\theta\notag\\&=\int_{[t_n+s,t_n-s]\times S^1}-\Delta \hat{u}_n (\hat{u}_n-\hat{u}_n^*) dtd\theta \notag\\
&=\int_{[t_n+s,t_n-s]\times S^1}\nabla \hat{u}_n\nabla (\hat{u}_n-\hat{u}_n^*)dtd\theta-\int_{\{t_n+s\}\times S^1}\frac{\partial \hat{u}_n}{\partial t}(\hat{u}_n-\hat{u}_n^*)d\theta+\int_{\{t_n-s\}\times S^1}\frac{\partial \hat{u}_n}{\partial t}(\hat{u}_n-\hat{u}_n^*)d\theta.
\end{align}

Using H\"{o}lder's inequality, we have
\begin{align}\label{inequ:07}
&\int_{[t_n+s,t_n-s]\times S^1}\nabla \hat{u}_n\nabla (\hat{u}_n-\hat{u}_n^*)dtd\theta\notag\\&=\int_{[t_n+s,t_n-s]\times S^1}\left(|\nabla \hat{u}_n|^2-\frac{\partial \hat{u}_n}{\partial t}\frac{\partial \hat{u}^*_n}{\partial t}\right)dtd\theta\notag\\
&\geq \int_{[t_n+s,t_n-s]\times S^1}|\nabla \hat{u}_n|^2dtd\theta-\left(\int_{[t_n+s,t_n-s]\times S^1}|\frac{\partial \hat{u}_n}{\partial t}|^2dtd\theta\right)^{\frac{1}{2}}\left(\int_{[t_n+s,t_n-s]\times S^1}|\frac{\partial \hat{u}_n^*}{\partial t}|^2dtd\theta\right)^{\frac{1}{2}}\notag\\ &\geq \int_{[t_n+s,t_n-s]\times S^1}\left(|\nabla \hat{u}_n|^2-|\frac{\partial \hat{u}_n}{\partial t}|^2\right)dtd\theta,
\end{align}
where the last inequality follows from the fact that
\begin{align*}
\left(\int_{[t_n+s,t_n-s]\times S^1}|\frac{\partial \hat{u}_n^*}{\partial t}|^2dtd\theta\right)^{\frac{1}{2}}&= \left(\int_{[t_n+s,t_n-s]\times S^1}\left|\frac{1}{2\pi}\int_0^{2\pi}\frac{\partial \hat{u}_n}{\partial t}d\theta\right|^2dtd\theta\right)^{\frac{1}{2}}\\
&\leq \left(\int_{[t_n+s,t_n-s]\times S^1}\frac{1}{2\pi}\int_0^{2\pi}|\frac{\partial \hat{u}_n}{\partial t}|^2d\theta dtd\theta\right)^{\frac{1}{2}}\\
&=\left(\int_{[t_n+s,t_n-s]\times S^1}|\frac{\partial \hat{u}_n}{\partial t}|^2dtd\theta\right)^{\frac{1}{2}}.
\end{align*}

Combining \eqref{equ:12} with \eqref{inequ:07} , we arrived at
\begin{align}\label{inequ:08}
&\int_{[t_n-s,t_n+s]\times S^1}\left(|\nabla \hat{u}_n|^2-|\frac{\partial \hat{u}_n}{\partial t}|^2\right)dtd\theta\notag\\ &\leq \int_{[t_n-s,t_n+s]\times S^1}\Upsilon_{\hat{u}_n}(\nabla \hat{u}_n,\nabla \hat{u}_n)(\hat{u}_n-\hat{u}_n^*)dtd\theta - \int_{[t_n-s,t_n+s]\times S^1} \hat{\Gamma}(\hat{u}_n) (\hat{u}_n-\hat{u}_n^*) dtd\theta\notag\\&\quad+\int_{\{t_n+s\}\times S^1}\frac{\partial \hat{u}_n}{\partial t}(\hat{u}_n-\hat{u}_n^*)d\theta-\int_{\{T_n-s\}\times S^1}\frac{\partial \hat{u}_n}{\partial t}(\hat{u}_n-\hat{u}_n^*)d\theta.
\end{align}

A direct computation yields that
\begin{align}
\int_{t_n-s}^{t_n+s}\int_\pi^{2\pi}\left(|\nabla \hat{u}_n|^2-|\frac{\partial \hat{u}_n}{\partial t}|^2\right)dtd\theta =\int_{t_n-s}^{t_n+s}\int^\pi_{0}\left(|D\sigma(u_n)\cdot\nabla u_n|^2-|D\sigma(u_n)\cdot\frac{\partial u_n}{\partial t}|^2\right)dtd\theta.
\end{align}
Note that
\begin{align*}
|D\sigma(u_n)\cdot\nabla u_n|^2&=\langle D\sigma\cdot\nabla u_n,D\sigma\cdot\nabla u_n\rangle\\&=\langle (D\sigma)^*\cdot D\sigma\cdot\nabla u_n,\nabla u_n\rangle\\&= \left\langle \bigg((D\sigma)^*\cdot D\sigma-Id\bigg)\cdot\nabla u_n,\nabla u_n\right\rangle+|\nabla u_n|^2,
\end{align*}
where $(D\sigma)^*$ is the adjoint operator of the linear operator $D\sigma:TN|_{K_{\delta_0}}\to TN|_{K_{\delta_0}}$.

Similarly, $$|D\sigma\cdot\frac{\partial u_n}{\partial t}|^2=\left\langle \bigg((D\sigma)^*\cdot D\sigma-Id\bigg)\cdot\frac{\partial u_n}{\partial t},\frac{\partial u_n}{\partial t}\right\rangle+|\frac{\partial u_n}{\partial t}|^2.$$

Noting that $(D\sigma)^*\cdot D\sigma|_K=Id$, by the continuity of eigenvalues of $(D\sigma)^*\cdot D\sigma$, we have that for any $\delta'>0$, there exists a constant  $\delta_1=\delta_1(\delta',K,N)$, such that for any $y\in K_{\delta_1}$ and $\xi\in TN|_{K_{\delta_1}}$, there holds $$\langle (D\sigma)^*|_{y}\cdot D\sigma|_y\cdot\xi, \xi\rangle\leq (1+\delta')|\xi|^2.$$

By Lemma \ref{lem:small-energy-regularity}, for any $\epsilon>0$, when $n$ is big enough, there holds $$\|dist(u_n,K)\|_{L^\infty([-T_n+1,T_n-1]\times (0,\pi))}\leq C\epsilon.$$ Thus, when $n$ is big enough, we have
$$\left\langle \bigg((D\sigma)^*|_{u_n(t,\theta)}\cdot D\sigma|_{u_n(t,\theta)}-Id\bigg)\cdot\xi, \xi\right\rangle\leq C\epsilon|\xi|^2, \quad  (t,\theta)\in [-T_n+1,T_n-1]\times (0,\pi),$$
which implies that
\begin{align*}
\int_{t_n+s}^{t_n-s}\int_\pi^{2\pi}\left(|\nabla \hat{u}_n|^2-|\frac{\partial \hat{u}_n}{\partial t}|^2\right)dtd\theta\geq \int_{t_n+s}^{t_n-s}\int^{\pi}_{0}\left(|\nabla u_n|^2-|\frac{\partial u_n}{\partial t}|^2\right)dtd\theta-C\epsilon \int_{t_n+s}^{t_n-s}\int^{\pi}_{0}|\nabla u_n|^2dtd\theta.
\end{align*}

Therefore, we have
\begin{align*}
&\int_{[t_n-s,t_n+s]\times S^1}\left(|\nabla \hat{u}_n|^2-|\frac{\partial \hat{u}_n}{\partial t}|^2\right)dtd\theta\\&=\int_{-T_n+1}^{T_n-1}\int^{\pi}_{0}\left(|\nabla u_n|^2-|\frac{\partial u_n}{\partial t}|^2\right)dtd\theta +\int_{t_n-s}^{t_n+s}\int_\pi^{2\pi}\left(|\nabla \hat{u}_n|^2-|\frac{\partial \hat{u}_n}{\partial t}|^2\right)dtd\theta \\&\geq 2\int_{t_n-s}^{t_n+s}\int^{\pi}_{0}\left(|\nabla u_n|^2-|\frac{\partial u_n}{\partial t}|^2\right)dtd\theta-C\epsilon \int_{t_n-s}^{t_n+s}\int^{\pi}_{0}|\nabla u_n|^2dtd\theta.
\end{align*}
Combining this with \eqref{inequ:08} and \eqref{inequ:09}, using Hölder's inequality and Poincaré's inequality, it yields
\begin{align}\label{inequ:01}
&2\int_{t_n-s}^{t_n+s}\int^{\pi}_{0}\left(|\nabla u_n|^2-|\frac{\partial u_n}{\partial t}|^2\right)dtd\theta\notag\\ &\leq \int_{[t_n-s,t_n+s]\times S^1}\Upsilon_{\hat{u}_n}(\nabla \hat{u}_n,\nabla \hat{u}_n)(\hat{u}_n-u_n^*)dtd\theta - \int_{[t_n-s,t_n+s]\times S^1} \hat{\Gamma}(\hat{u}_n) (\hat{u}_n-u_n^*) dtd\theta\notag\\&\quad+C\epsilon \int_{t_n-s}^{t_n+s}\int^{\pi}_{0}|\nabla u_n|^2dtd\theta+\int_{\{t_n+s\}\times S^1}\frac{\partial \hat{u}_n}{\partial t}(\hat{u}_n-u_n^*)d\theta-\int_{\{t_n-s\}\times S^1}\frac{\partial \hat{u}_n}{\partial t}(\hat{u}_n-u_n^*)d\theta\notag\\ &\leq C\epsilon \int_{t_n-s}^{t_n+s}\int^{\pi}_{0}|\nabla u_n|^2dtd\theta+C \left(\int_{t_n-s}^{t_n+s}\int^{\pi}_{0}|\tau(u_n)|^2dtd\theta\right)^{\frac{1}{2}} \left(\int_{t_n-s}^{t_n+s}\int^{\pi}_{0}|\frac{\partial u_n}{\partial\theta}|^2dtd\theta\right)^{\frac{1}{2}}\notag\\&\quad +\int_{\{t_n+s\}\times S^1}\frac{\partial \hat{u}_n}{\partial t}(\hat{u}_n-u_n^*)d\theta-\int_{\{t_n-s\}\times S^1}\frac{\partial \hat{u}_n}{\partial t}(\hat{u}_n-u_n^*)d\theta.
\end{align}
For above boundary terms, by Poincaré's inequality, we get
\begin{align}\label{inequ:02}
\left|\int_{\{t_n+s\}\times S^1}\frac{\partial \hat{u}_n}{\partial t}(\hat{u}_n-u_n^*)d\theta\right|&\leq \left(\int_{\{t_n+s\}\times S^1}|\frac{\partial \hat{u}_n}{\partial t}|^2d\theta\right)^{\frac{1}{2}}\left(\int_{\{t_n+s\}\times S^1}|\hat{u}_n-u_n^*|^2d\theta\right)^{\frac{1}{2}}\notag\\
&\leq \left(\int_{\{t_n+s\}\times S^1}|\frac{\partial \hat{u}_n}{\partial t}|^2d\theta\right)^{\frac{1}{2}}\left(\int_{\{t_n+s\}\times S^1}|\frac{\partial \hat{u}_n}{\partial \theta}|^2d\theta\right)^{\frac{1}{2}}\notag\\
&\leq \int_{\{t_n+s\}\times S^1}|\nabla \hat{u}_n|^2d\theta\leq 3\int_{\{t_n+s\}\times (0,\pi)}|\nabla u_n|^2d\theta.
\end{align}
Similarly,
\begin{equation}\label{inequ:03}
\left|\int_{\{t_n-s\}\times S^1}\frac{\partial \hat{u}_n}{\partial t}(\hat{u}_n-\hat{u}_n^*)d\theta\right|\leq 3\int_{\{t_n-s\}\times (0,\pi)}|\nabla u_n|^2d\theta.\end{equation}

Now taking $t_n=0$ and $s=T_n-1$ in \eqref{inequ:01}, when $n$ is big enough, we get
\begin{align*}
&\int_{-T_n+1}^{T_n-1}\int^{\pi}_{0}\left(|\nabla u_n|^2-|\frac{\partial u_n}{\partial t}|^2\right)dtd\theta\\
&\leq C\epsilon+C\|\tau(u_n)\|_{L^2(Q_n)}+C\int_{\{\{-T_n+1,T_n-1\}\}\times (0,\pi)}|\nabla u_n|^2d\theta\\
&\leq C\epsilon+C\|\tau(u_n)\|_{L^2(Q_n)}+CE(u_n;[-T_n,-T_n+2]\times (0,\pi))+CE(u_n;[T_n-2,T_n]\times (0,\pi))\\&\leq C\epsilon,
\end{align*} which means that $\lim_{n\to\infty}\int_{-T_n+1}^{T_n-1}\int^{\pi}_{0}|\frac{\partial u_n}{\partial \theta}|^2dtd\theta=0$.

Now, we consider two cases in the theorem. If we assume that $$\lim_{n\to\infty}\|\tau(u_n)\cdot \partial_t u_n\|_{L^1(Q_n)}\cdot T_n=0,$$ then
\begin{align*}
E(u_n;[-T_n+1,T_n-1]\times [0,\pi])&=\int_{-T_n+1}^{T_n-1}\int^{\pi}_{0}\left(\left|\frac{\partial u_n}{\partial t}\right|^2-\left|\frac{\partial u_n}{\partial \theta}\right|^2\right)dtd\theta+o(1)\\
&=\int_{-T_n+1}^{T_n-1}\left(\alpha_n+2\int_0^s\int^{\pi}_{0}\tau(u_n)\cdot \partial_t u_nd\theta dt\right)ds+o(1)\\&=2\alpha_n T_n+o(1),
\end{align*} where we used the fact that $$\left|\int_{-T_n+1}^{T_n-1}\int_0^s\int^{\pi}_{0}\tau(u_n)\cdot \partial_t u_nd\theta dtds\right|\leq 2\|\tau(u_n)\cdot \partial_t u_n\|_{L^1(Q_n)}\cdot T_n=o(1).$$ This yields the first conclusion of theorem.

If we assume that $$\lim_{n\to\infty}\|\ |\phi(u_n)|\ \|_{L^1(Q_n)}=0,$$ then
\begin{align*}
E(u_n;[-T_n+1,T_n-1]\times [0,\pi])&=\int_{-T_n+1}^{T_n-1}\int^{\pi}_{0}\left(\left|\frac{\partial u_n}{\partial t}\right|^2-\left|\frac{\partial u_n}{\partial \theta}\right|^2\right)dtd\theta+o(1)\\
&\leq \|\ |\phi(u_n)|\ \|_{L^1(Q_n)}+o(1)=o(1),
\end{align*} which immediately yields the second conclusion. We finished the proof of this theorem.
\end{proof}

From the proof of above theorem, we can further get the following energy decay lemma
\begin{lem}\label{lem:04}
Under the assumptions of Theorem \ref{thm:01}, for any $t_n\in [-T_n+1,T_n-1]$, denoting $d_n:=\min\{t_n-(-T_n),T_n-t_n\}$, we have
\begin{align*}
E(u_n;[t_n-1,t_n+1]\times (0,\pi))&\leq CE(u_n;Q_n)e^{-\frac{d_n}{6}} + C|\alpha_n|+C\int_{Q_n}|\tau(u_n)\cdot \partial_t u_n | dtd\theta  +C\|\tau(u_n)\|^2_{L^2(Q_n)}.
 \end{align*} or respectively
 \begin{align*}
E(u_n;[t_n-1,t_n+1]\times S^1)&\leq CE(u_n;P_n)e^{-\frac{d_n}{6}} + C|\alpha_n|+C\int_{P_n}|\tau(u_n)\cdot \partial_t u_n | dtd\theta  +C\|\tau(u_n)\|^2_{L^2(P_n)}.
 \end{align*}
\end{lem}
\begin{proof}
We just show the proof for the case of $Q_n$.

Denoting $$f(s):=\int_{t_n-s}^{t_n+s}\int_0^{\pi}|\nabla u_n|^2d\theta dt.$$ By \eqref{inequ:01}, \eqref{inequ:02}, \eqref{inequ:03},  using Hölder's inequality, we obtain that
\begin{align}\label{inequ:10}
&\int_{t_n-s}^{t_n+s}\int_0^\pi|\nabla u_n|^2dtd\theta+\int_{t_n-s}^{t_n+s}\int_0^\pi\left(\left|\frac{\partial u_n}{\partial\theta}\right|^2- \left|\frac{\partial u_n}{\partial t}\right|^2\right)dtd\theta\notag\\
&\leq C\epsilon\int_{t_n-s}^{t_n+s}\int_0^\pi|\nabla u_n|^2dtd\theta+\frac{1}{8}\int_{t_n-s}^{t_n+s}\int_0^\pi|\nabla u_n|^2dtd\theta\notag\\&\quad +C\|\tau(u_n)\|^2_{L^2(Q_n)}+3\int_{\{t_n-s,t_n+s\}\times (0,\pi)}|\nabla u_n|^2d\theta
\end{align}
Taking $C\epsilon\leq \frac{1}{8}$ in above inequality and using Hölder's inequality, we get
\begin{align*}
\frac{3}{4}f(s)&\leq  3f'(s) +C\|\tau(u_n)\|^2_{L^2(Q_n)}+\int_{t_n-s}^{t_n+s}\int_0^\pi\left( \left|\frac{\partial u_n}{\partial t}\right|^2-\left|\frac{\partial u_n}{\partial\theta}\right|^2\right)dtd\theta\\&\leq 3f'(s)+C\|\tau(u_n)\|^2_{L^2(Q_n)}+2s\alpha_n +2\int_{t_n-s}^{t_n+s}\int_{[0,\xi]\times [0,\pi]}\tau(u_n)\cdot \partial_t u_n dtd\theta d\xi\\&\leq 3f'(s)+C\|\tau(u_n)\|^2_{L^2(Q_n)}+2s\left(|\alpha_n|+2\int_{Q_n}|\tau(u_n)\cdot \partial_t u_n | dtd\theta \right)
\end{align*} which implies
\begin{equation}
\left(e^{-\frac{1}{4}s}f(s)\right)'\geq -\left(|\alpha_n|+2\int_{Q_n}|\tau(u_n)\cdot \partial_t u_n | dtd\theta \right)  se^{-\frac{1}{4}s}-C\|\tau(u_n)\|^2_{L^2(Q_n)} e^{-\frac{1}{4}s}
\end{equation}

Integrating from $1$ to $L=d_n:=\min\{t_n-(-T_n),T_n-t_n\}$, we get
\begin{align*}
f(1)&=\int_{t_n-1}^{t_n+1}\int_0^{\pi}|\nabla u_n|^2d\theta dt\\&\leq CE(u_n;Q_n)e^{-\frac{d_n}{6}} + C|\alpha_n|+C\int_{Q_n}|\tau(u_n)\cdot \partial_t u_n | dtd\theta  +C\|\tau(u_n)\|^2_{L^2(Q_n)}.
\end{align*}
\end{proof}

\begin{prop}\label{prop:03}
Under the assumptions of Theorem \ref{thm:01}, suppose $$\lim_{n\to\infty}\left(\|\tau(u_n)\|^2_{L^2(Q_n)}+\|\tau(u_n)\cdot \partial_t u_n\|_{L^1(Q_n)}\right)\cdot T^2_n=0$$ or respectively $$\lim_{n\to\infty}\left(\|\tau(u_n)\|^2_{L^2(P_n)}+\|\tau(u_n)\cdot \partial_t u_n\|_{L^1(P_n)}\right)\cdot T^2_n=0.$$ Then if $\lim_{n\to\infty} (\sqrt{|\alpha_n|}T_n)=0,$ we have $$\lim_{n\to\infty}osc_{Q_n}u_n=0\ \ or\ \ \lim_{n\to\infty}osc_{P_n}u_n=0.$$
\end{prop}
\begin{proof}
By Theorem \ref{thm:01} and assumption that $\lim_{n\to\infty} (\sqrt{|\alpha_n|}T_n)=0$, we first have that $$\lim_{n\to\infty}E(u_n;Q_n)=0.$$ Denoting $Q_n^i:=[-T_n+i,-T_n+i+1]\times [0,\pi]$, $i=1,2,...,2T_n-2$, by Lemma \ref{lem:small-energy-regularity-1}, Lemma \ref{lem:small-energy-regularity} and Lemma \ref{lem:04}, we get
\begin{align*}
osc_{Q_n^i}u_n&\leq C(\|\nabla u_n\|_{L^2(Q_n^{i-1}\cup Q_n^i\cup Q_n^{i+1})}+\|\tau(u_n)\|_{L^2(Q_n^{i-1}\cup Q_n^i\cup Q_n^{i+1})})\\
&\leq C\sqrt{E(u_n;Q_n)}\left(e^{-\frac{i}{12}}+e^{-\frac{2T_n-i}{12}}\right) + C\left(|\alpha_n|+\int_{Q_n}|\tau(u_n)\cdot \partial_t u_n | dtd\theta  +\|\tau(u_n)\|^2_{L^2(Q_n)}\right)^{\frac{1}{2}}
\end{align*} which implies that
\begin{align*}
& osc_{[-T_n+1,T_n-1]\times [0,\pi]}u_n\leq \sum_{i=1}^{2T_n-2}osc_{Q_n^i}u_n\\&\leq C\sqrt{E(u_n;Q_n)}\sum_{i=1}^{2T_n-2}\left(e^{-\frac{i}{12}}+e^{-\frac{2T_n-i}{12}}\right) + C\left(|\alpha_n|+\int_{Q_n}|\tau(u_n)\cdot \partial_t u_n | dtd\theta  +\|\tau(u_n)\|^2_{L^2(Q_n)}\right)^{\frac{1}{2}}\cdot T_n\\
&=o(1).
\end{align*}
Then by \eqref{equ:03} and Lemma \ref{lem:small-energy-regularity}, we immediately have $\lim_{n\to\infty} osc_{Q_n}u_n=0$.
\end{proof}

\

\section{Refined neck analysis: convergence to a geodesic or geodesic-like curve}\label{sec:refined-analysis}

\

In this section, we will apply a refined neck analysis and prove that the limits of necks are either some geodesic or geodesic-like curves.  Theorem \ref{thm-main:01} and Theorem \ref{thm-main:03} will be proved in this section.

Define $$\mu:=\lim_{n\to\infty}\sqrt{|\alpha_n|}T_n.$$
In order to prove Theorem \ref{thm-main:01} and Theorem \ref{thm-main:03}, we first show the following results
\begin{thm}\label{thm:02}
Let $N$ be a compact Riemannian manifold and $K\subset N$ is a smooth submanifold. The following two alternatives hold.
\begin{itemize}
\item[(1)] Let  $u_n:P_n\to N$ be a sequence of maps from  cylinder $P_n=[-T_n,T_n]\times S^1$  with free boundary $u_n(\{-T_n,T_n\}\times S^1)$ on $K$, satisfying  $$E(u_n;P_n)\leq\Lambda,\ \ \lim_{n\to\infty}\left(\|\tau(u_n)\|^2_{L^2(P_n)}+\|\tau(u_n)\cdot \partial_t u_n\|_{L^1(P_n)}\right)\cdot T^2_n=0.$$ If there is no energy concentration for $u_n$, i.e. \eqref{equ:04} holds, then we have:

    \

    \begin{itemize}
    \item[(1-1)] when $0\leq \mu<\infty$, $u_n$ converges to a geodesic of length $\sqrt{\frac{2}{\pi}}\mu$.

    \

    \item[(1-2)] when $\mu=\infty$, the neck contains at least an infinite length geodesic.
    \end{itemize}

\

   \item[(2)] Let  $u_n:Q_n\to N$ be a sequence of maps from vertical half cylinder $Q_n:=[-T_n,T_n]\times [0,\pi]$ with free boundary $u_n([-T_n,T_n]\times \{0,\pi\})$ on $K$, satisfying $$E(u_n;Q_n)\leq\Lambda,\ \ \lim_{n\to\infty}\left(\|\tau(u_n)\|^2_{L^2(Q_n)}+\|\tau(u_n)\cdot \partial_t u_n\|_{L^1(Q_n)}\right)\cdot T^2_n=0.$$ If there is no energy concentration for $u_n$, i.e. \eqref{equ:03} holds, then  we have:

\

    \begin{itemize}
    \item[(2-1)] when $0\leq \mu<\infty$, $u_n$ converges to a geodesic-like curve (see definition \eqref{def:geodesic-like-curve}) of length $\frac{2}{\sqrt{\pi}}\mu$.

    \

    \item[(2-2)] when $\mu=\infty$, the neck contains at least an infinite length geodesic-like curve.
    \end{itemize}

\end{itemize}
\end{thm}

\

For the proof of Theorem \ref{thm:02}, we just show the details for the second case since the first one is similar and easier.

\

We start by proving some lemmas.

\begin{lem}\label{lem:02}
Under assumptions of Theorem \ref{thm:02}, if $\mu>0,$ then for any $k>0$ and $t_n\in [- T_n+T_n^{\frac{1}{4}}, T_n-T_n^{\frac{1}{4}}]$, there holds
\begin{align*}
\lim_{n\to\infty}\frac{1}{|\alpha_n|}\int_{[t_n-k,t_n+k]\times [0,\pi]}\left|\frac{\partial u_n}{\partial\theta}\right|^2dtd\theta=0.
\end{align*}
\end{lem}

\begin{proof}
Denoting $$\beta_n:=\sup_{-T_n\leq t\leq T_n-1}E(u_n, [t,t+1]\times [0,\pi]).$$ Then \eqref{equ:03} implies  $\lim_{n\to\infty}\beta_n=0$. Next, we divide the proof into two steps.

\

\noindent\textbf{Step 1:} We firstly claim that for any $s_n>0$ such that $$\lim_{n\to\infty}s_n=+\infty,\ \ s_n=o(1)T^{\frac{1}{4}}_n,\ \ s_n\beta_n=o(1),$$ there holds
\begin{equation}\label{equ:05}
\frac{1}{|\alpha_n|}\int_{[t_n-s_n,t_n+s_n]\times [0,\pi]}\left|\frac{\partial u_n}{\partial\theta}\right|^2dtd\theta\leq C,
\end{equation} where $C>0$ is independent of $n$.

\

In fact, since $\mu:=\lim_{n\to\infty} (\sqrt{\alpha_n}T_n)>0,$ by Lemma \ref{lem:04}, we get
\begin{align}\label{inequ:06}
\frac{1}{|\alpha_n|}\int_{[t_n-\sigma,t_n+\sigma]\times [0,\pi]}\left|\nabla u_n\right|^2dtd\theta &\leq C\sigma e^{-\frac{1}{6}T^{\frac{1}{4}}_n}\cdot \frac{1}{|\alpha_n|}+C\sigma+C\frac{\|\tau(u_n)\cdot \partial_tu_n\|_{L^1(Q_n)}+\|\tau(u_n)\|^2_{L^2(Q_n)}}{|\alpha_n|}\notag \\& \leq C\sigma e^{-\frac{1}{6}T^{\frac{1}{4}}_n}\cdot T_n^2+C\sigma+C\left(\|\tau(u_n)\cdot \partial_tu_n\|_{L^1(Q_n)}+\|\tau(u_n)\|^2_{L^2(Q_n)}\right) T_n^2\notag\\&\leq C(1+\sigma).
\end{align}

By \eqref{inequ:01}-\eqref{inequ:03}, \eqref{inequ:06}, using Hölder's inequality and the the fact that
\begin{equation}\label{inequ:04}
\frac{\|\tau(u_n)\cdot \partial_tu_n\|_{L^1(Q_n)}+\|\tau(u_n)\|^2_{L^2(Q_n^+)}}{|\alpha_n|}\leq C\|\tau(u_n)\|^2_{L^2(Q_n^+)}T_n^2=o(1),
\end{equation}
 we obtain
\begin{align}\label{inequ:05}
&\int_{t_n-s_n}^{t_n+s_n}\int^{\pi}_{0}\left(|\nabla u_n|^2-|\frac{\partial u_n}{\partial t}|^2\right)dtd\theta\notag\\  &\leq C\beta_n \int_{t_n-s_n}^{t_n+s_n}\int^{\pi}_{0}|\nabla u_n|^2dtd\theta+C \left(\int_{t_n-s_n}^{t_n+s_n}\int^{\pi}_{0}|\tau(u_n)|^2dtd\theta\right)^{\frac{1}{2}} \left(\int_{t_n-s_n}^{t_n+s_n}\int^{\pi}_{0}|\frac{\partial u_n}{\partial \theta}|^2dtd\theta\right)^{\frac{1}{2}} \notag\\&\quad +C\int_{[t_n+s_n-1,t_n+s_n+1]\times [0,\pi]}|\nabla u_n|^2dtd\theta+ C\int_{[t_n-s_n-1,t_n-s_n+1]\times [0,\pi]}|\nabla u_n|^2dtd\theta\notag
\\ &\leq \frac{1}{2}\int_{t_n-s_n}^{t_n+s_n}\int^{\pi}_{0}|\frac{\partial u_n}{\partial \theta}|^2dtd\theta +C(|\alpha_n|+\|\tau(u_n)\|_{L^2}^2)(\beta_n  s_n+1)\notag\\&\leq \frac{1}{2}\int_{t_n-s_n}^{t_n+s_n}\int^{\pi}_{0}|\frac{\partial u_n}{\partial \theta}|^2dtd\theta+C (|\alpha_n|+\|\tau(u_n)\|_{L^2}^2),
\end{align} which implies \eqref{equ:05}.

\

\noindent\textbf{Step 2:} We prove that for any positive constant $k>0$, there holds
\begin{align}\label{equ:10}
\lim_{n\to\infty}\frac{1}{|\alpha_n|}\int_{[t_n-k,t_n+k]\times [0,\pi]}\left|\frac{\partial u_n}{\partial\theta}\right|^2dtd\theta=0.
\end{align}

\

By \eqref{equ:05}, \eqref{inequ:04} and Fubini's theorem, one can see that there exists $s_n'\in (\frac{1}{2}s_n,s_n)$ such that
\begin{equation}\label{equ:06}
\frac{1}{|\alpha_n|}\int_{\{t_n-s'_n\}\times [0,\pi]}\left|\frac{\partial u_n}{\partial\theta}\right|^2d\theta+ \frac{1}{|\alpha_n|}\int_{\{t_n+s'_n\}\times [0,\pi]}\left|\frac{\partial u_n}{\partial\theta}\right|^2d\theta=o(1)
\end{equation}

By the definition of $\alpha_n$ and Hölder's inequality, for ant $t\in [-T_n,T_n]$, we have
\begin{align*}
\int_{\{t\}\times (0,\pi)}|\partial_t u_n|^2d\theta&\leq |\alpha_n|+\int_{\{t\}\times [0,\pi]}|\partial_\theta u_n|^2d\theta+ 2\left|\int_{[0,t]\times (0,\pi)}\partial_t u_n\cdot \tau(u_n)d\theta\right|\\
&= |\alpha_n|+\int_{\{t\}\times (0,\pi)}|\partial_\theta u_n|^2d\theta+ o(1)|\alpha_n|
\end{align*} which implies
\begin{align}\label{equ:08}
\left|\int_{\{t_n+s_n'\}\times S^1}\frac{\partial \hat{u}_n}{\partial t}(\hat{u}_n-\hat{u}_n^*)d\theta\right| &\leq \left(\int_{\{t_n+s_n'\}\times S^1}|\frac{\partial \hat{u}_n}{\partial t}|^2d\theta\right)^{\frac{1}{2}}\left(\int_{\{t_n+s_n'\}\times S^1}|\hat{u}_n-\hat{u}_n^*|^2d\theta\right)^{\frac{1}{2}}\notag\\
&\leq \left(\int_{\{t_n+s_n'\}\times S^1}|\frac{\partial \hat{u}_n}{\partial t}|^2d\theta\right)^{\frac{1}{2}}\left(\int_{\{t_n+s_n'\}\times S^1}|\frac{\partial \hat{u}_n}{\partial \theta}|^2d\theta\right)^{\frac{1}{2}}\notag\\
&\leq C\left(\int_{\{t_n+s_n'\}\times (0,\pi)}|\frac{\partial u_n}{\partial t}|^2d\theta\right)^{\frac{1}{2}}\left(\int_{\{t_n+s_n'\}\times (0,\pi)}|\frac{\partial u_n}{\partial \theta}|^2d\theta\right)^{\frac{1}{2}}\notag\\
&\leq C\left(\int_{\{t_n+s_n'\}\times (0,\pi)}|\frac{\partial u_n}{\partial \theta}|^2d\theta+|\alpha_n|\right)^{\frac{1}{2}}\left(\int_{\{t_n+s_n'\}\times (0,\pi)}|\frac{\partial u_n}{\partial \theta}|^2d\theta\right)^{\frac{1}{2}}=o(1)|\alpha_n|.
\end{align}

Similarly,
\begin{equation}\label{equ:09}
\left|\int_{\{t_n+s_n'\}\times S^1}\frac{\partial \hat{u}_n}{\partial t}(\hat{u}_n-\hat{u}_n^*)d\theta\right|=o(1)|\alpha_n|.
\end{equation}

Combining \eqref{inequ:01} with \eqref{inequ:06},\eqref{inequ:04}, \eqref{equ:08} and \eqref{equ:09}, we obtain
\begin{align*}
&\int_{t_n-s_n'}^{t_n+s_n'}\int_0^\pi |\frac{\partial u_n}{\partial\theta}|^2d\theta dt \\&\leq C\beta_n s_n |\alpha_n|+ C\|\tau(u_n)\|_{L^2}^2+\left|\int_{\{t_n+s_n'\}\times S^1}\frac{\partial \hat{u}_n}{\partial t}(\hat{u}_n-\hat{u}_n^*)d\theta\right|+\left|\int_{\{t_n-s_n'\}\times S^1}\frac{\partial \hat{u}_n}{\partial t}(\hat{u}_n-\hat{u}_n^*)d\theta \right| \\ &= o(1)|\alpha_n|
\end{align*} which implies \eqref{equ:10}.
\end{proof}

\begin{lem}\label{lem:03}
Under assumptions of Lemma \ref{lem:02}, for any $t_n\in [-T_n+T_n^{\frac{1}{4}},\lambda T_n-T_n^{\frac{1}{4}}]$ and fixed positive constant $k>0$, we have
\begin{align}
\frac{1}{\sqrt{|\alpha_n|}}\big(\hat{u}_n(t,\theta)-\hat{u}_n(t_n,0)\big)   \to \overrightarrow{a}t,\ \ weakly\ \ in\ \ W^{2,2}([t_n-k,t_n+k]\times S^1)
\end{align} where $|\overrightarrow{a}|=\frac{1}{\sqrt{\pi}}$ \footnote{Here, for the case of $P_n$, one can check that $|\overrightarrow{a}|=\frac{1}{\sqrt{2\pi}}$.} and
\begin{align}
\frac{1}{\sqrt{|\alpha_n|}}\left(\hat{u}_n(t,\theta)-\frac{1}{2\pi}\int_0^{2\pi}\hat{u}_n(t,\theta)d\theta\right)   \to 0,\ \ weakly\ \ in\ \ W^{2,2}([t_n-k,t_n+k]\times S^1).
\end{align}

\end{lem}
\begin{proof}
We just prove the first conclusion, since the second one is similar.

\

By \eqref{inequ:06}, \eqref{inequ:04} and Lemma \ref{lem:small-energy-regularity}, we have
\begin{align*}
\|\nabla u_n\|_{W^{1,2}([t_n-k,t_n+k]\times [0,\pi])}\leq C(k)\sqrt{|\alpha_n|}.
\end{align*}

Setting $$\hat{v}_n(t,\theta):=\frac{1}{\sqrt{|\alpha_n|}}\big(\hat{u}_n(t_n+t,\theta)-\hat{u}_n(t_n,0)\big), $$
then there holds
 $$\|\nabla \hat{v}_n\|_{W^{1,2}([-k,+k]\times S^1)}\leq C(k).$$
 Noting that $\hat{v}_n(0,0)=0$, then by embedding theory, we have $$\|\hat{v}_n\|_{C^{0}([-k,+k]\times S^1)}\leq C(k).$$

Since $\hat{v}_n$ satisfies following equation $$\Delta \hat{v}_n+\sqrt{|\alpha_n|}\Upsilon(\nabla \hat{v}_n, \nabla \hat{v}_n)-\frac{1}{\sqrt{|\alpha_n|}}\hat{\Gamma}=0,$$ noting that $$\frac{\|\hat{\Gamma}\|_{L^2([-k,k]\times S^1)}}{\sqrt{|\alpha_n|}}\leq C\|\tau(u_n)\|_{L^2}\cdot T_n=o(1),$$ passing to a subsequence, we know that $\hat{v}_n\rightharpoonup v$ weakly in $W^{2,2}_{loc}(\R\times S^1)$ where $v$ satisfies $$\Delta v=0.$$

By Lemma \ref{lem:02}, we get $$\frac{\partial v}{\partial\theta}=0.$$
Therefore, $v$ must be of the following form
$$v=\overrightarrow{a}t$$
for some constant vector $\overrightarrow{a}\in T_yN$.

On one hand, by compact embedding theory, we have $$\lim_{n\to\infty}\int_{-1}^{1}\int_0^\pi|\nabla \hat{v}_n|^2d\theta dt= \int_{-1}^{1}\int_0^\pi|\nabla v|^2d\theta dt=2\pi |\overrightarrow{a}|^2.$$
On the other hand, Lemma \ref{lem:02} tells us
\begin{align*}
\lim_{n\to\infty}\frac{1}{|\alpha_n|}\int_{[t_n-k,t_n+k]\times [0,\pi]}\left|\frac{\partial u_n}{\partial\theta}\right|^2dtd\theta=0,
\end{align*} it yields that
\begin{align*}
\lim_{n\to\infty}\int_{-1}^{1}\int_0^\pi|\nabla \hat{v}_n|^2d\theta dt&= \lim_{n\to\infty}\frac{1}{|\alpha_n|}\int_{[t_n-1,t_n+1]\times [0,\pi]}\left(\left|\frac{\partial u_n}{\partial t}\right|^2+\left|\frac{\partial u_n}{\partial\theta}\right|^2\right)dtd\theta\\ &= \lim_{n\to\infty}\frac{1}{|\alpha_n|}\int_{[t_n-1,t_n+1]\times [0,\pi]}\left(\left|\frac{\partial u_n}{\partial t}\right|^2-\left|\frac{\partial u_n}{\partial\theta}\right|^2\right)dtd\theta\\ &=
\frac{2\alpha_n}{|\alpha_n|}+\frac{1}{|\alpha_n|}\int_{t_n-1}^{t_n+1}\int_{[0,s]\times [0,\pi]}\left(2\tau(u_n)\cdot \frac{\partial u_n}{\partial t}\right)dtd\theta ds=2,
\end{align*} where the last equality follows from \eqref{inequ:04}.

Then, we can see that $$|\overrightarrow{a}|=\frac{1}{\sqrt{\pi}}.$$
\end{proof}

Now we begin to prove Theorem \ref{thm:02}.

\begin{proof}[\textbf{Proof of Theorem \ref{thm:02}:}] We first consider the case that $0< \mu<\infty$.

\

Denote the curve $$\gamma_n: \ \ u_n^*(t)=\frac{1}{2\pi}\int_0^{2\pi}\hat{u}_n(t,\theta)d\theta,\ \ t\in [- T_n+T_n^{\frac{1}{4}}, T_n-T_n^{\frac{1}{4}}].$$

Denoting
\[
\dot{u}_n^*=\frac{du_n^*}{dt},\  \ddot{u}_n^*=\frac{d^2u_n^*}{dt^2},
\]
by a direct computation, we have
\begin{align*}
\ddot{u}_n^*(t)&=\frac{1}{2\pi}\int_0^{2\pi}\frac{\partial^2\hat{u}_n}{\partial t^2}(t,\theta)d\theta=\frac{1}{2\pi}\int_0^{2\pi}\Delta \hat{u}_nd\theta\\
&=-\frac{1}{2\pi}\int_0^{2\pi}A(\hat{u}_n)(\nabla \hat{u}_n,\nabla\hat{u}_n)d\theta+\frac{1}{2\pi}\int_\pi^{2\pi}D^2\sigma|_{\sigma(\hat{u}_n)}(D\sigma|_{\hat{u}_n}\cdot D\hat{u}_n, D\sigma|_{\hat{u}_n}\cdot D\hat{u}_n) d\theta\\
&\quad+\frac{1}{2\pi}\int_0^{2\pi}\hat\Gamma d\theta.
\end{align*}

Let $s$ be the arc length parameter of $\gamma_n$, i.e. $$s(t)=\int_{0}^t|\dot u_n^*(\xi)|d\xi.$$

Without loss of generality, we assume $2( T_n-T_n^{\frac{1}{4}})=k_n$ is an integer. By Lemma \ref{lem:03}, we have
\begin{align*}
\int_{- T_n+T_n^{\frac{1}{4}}}^{\lambda T_n-T_n^{\frac{1}{4}}}|\dot u_n^*(t)|dt&=\sum_{i=0}^{k_n-1}\int_{- T_n+T_n^{\frac{1}{4}}+i}^{- T_n+T_n^{\frac{1}{4}}+i+1}|\dot u_n^*(t)|dt\\
&=k_n \sqrt{|\alpha_n|}(\frac{1}{\sqrt{\pi}}+o(1))\\
&=\sqrt{|\alpha_n|}\ 2( T_n-T_n^{\frac{1}{4}})(\frac{1}{\sqrt{\pi}}+o(1))\to \frac{2 \mu}{\pi}
\end{align*} as $n\to\infty$.

Thus for any $s\in (0,\frac{2 \mu}{\pi})$,  there holds that $$u_n^*(s)|_{[-s,s]}\subset u_n^*(t)|_{[ T_n-T_n^{\frac{1}{4}},T_n-T_n^{\frac{1}{4}}]}$$ when $n$ is big enough.

By Lemma \ref{lem:small-energy-regularity-1}, it is easy to see that $\gamma_n:\ u_n^*(s)$ converges in $C^0([0,s_1])$ to a curve on $N$ denoted by $\gamma: u^*(s) $. Next, we will show that $\gamma$ is a geodesic-like curve.

Computing directly, we obtain  the following equations  in the sense of  almost everywhere, i.e.
\begin{align}
\frac{d^2u_n^*}{ds^2}
&=\frac{1}{|\dot{u}_n^*|^2}(\ddot{u}_n^*-\frac{\langle \ddot{u}_n^*,\dot{u}_n^*\rangle}{|\dot{u}_n^*|^2}\dot{u}_n^*)\notag\\
&=\frac{1}{|\dot{u}_n^*|^2}\frac{1}{2\pi}\left(-\int_0^{2\pi}A(\hat{u}_n)(\nabla \hat{u}_n,\nabla\hat{u}_n) d\theta+\int_\pi^{2\pi}D^2\sigma|_{\sigma(\hat{u}_n)}(D\sigma|_{\hat{u}_n}\cdot D\hat{u}_n, D\sigma|_{\hat{u}_n}\cdot D\hat{u}_n) d\theta\right)\notag \\
&\quad -\frac{\dot{u}_n^*}{|\dot{u}_n^*|^4}\frac{1}{2\pi}\bigg(-\int_0^{2\pi}\langle A(\hat{u}_n)(\nabla \hat{u}_n,\nabla\hat{u}_n),\dot{u}_n^* \rangle d\theta\notag\\& \quad +\int_\pi^{2\pi}\langle D^2\sigma|_{\sigma(\hat{u}_n)}(D\sigma|_{\hat{u}_n}\cdot D\hat{u}_n, D\sigma|_{\hat{u}_n}\cdot D\hat{u}_n), \dot{u}_n^*\rangle d\theta\bigg)\notag\\
&\quad +\frac{1}{|\dot{u}_n^*|^2}\frac{1}{2\pi}\int_0^{2\pi}\hat\Gamma d\theta -\frac{\dot{u}_n^*}{|\dot{u}_n^*|^4}\frac{1}{2\pi}\int_0^{2\pi}\langle \hat\Gamma, \dot{u}_n^* \rangle d\theta\notag \\
&=\mathbf{I}+\mathbf{II}+\mathbf{III}+\mathbf{IV}+\mathbf{V}+\mathbf{VI}.
\end{align}
By Lemma  \ref{lem:03}, we know that $$\frac{\dot{u}_n^*(t)}{\sqrt{|\alpha_n|}}\rightharpoonup \overrightarrow{a},\ \ weakly\  \ in\ \ W^{1,2}([t_n-k,t_n+k])$$ for any positive constant $k>0$, any $t_n\in [- T_n+T_n^{\frac{1}{4}}, T_n-T_n^{\frac{1}{4}}]$, where $|\overrightarrow{a}|=\frac{1}{\sqrt{\pi}}$. Embedding theory implies that $$\frac{\dot{u}_n^*(t)}{\sqrt{|\alpha_n|}}\rightarrow \overrightarrow{a},\ \ strongly\  \ in\ \ C^{0}([- T_n+T_n^{\frac{1}{4}}, T_n-T_n^{\frac{1}{4}}]).$$
By \eqref{inequ:06}, we have
\begin{align*}
\int_0^{s_1}|\mathbf{I}|ds&=\int_0^{s_1}\left|  \frac{1}{|\dot{u}_n^*|^2}\frac{1}{2\pi}\int_0^{2\pi}A(\hat{u}_n)(\nabla \hat{u}_n,\nabla\hat{u}_n) d\theta  \right|ds\\
&\leq \int_{- T_n+T_n^{\frac{1}{4}}}^{ T_n-T_n^{\frac{1}{4}}}\left|  \frac{1}{|\dot{u}_n^*|^2}\frac{1}{2\pi}\int_0^{2\pi}A(\hat{u}_n)(\nabla \hat{u}_n,\nabla\hat{u}_n) d\theta  \right||\dot{u}_n^*|dt\\
&\leq \frac{C}{\sqrt{|\alpha_n|}}\int_{- T_n+T_n^{\frac{1}{4}}}^{ T_n-T_n^{\frac{1}{4}}}\int_0^{\pi}|\nabla \hat{u}_n|^2 d\theta  dt\leq C T_n\sqrt{|\alpha_n|}\leq C
\end{align*} and

\begin{align}\label{inequ:14}
\int_0^{s_1}|\mathbf{V}|ds&=\int_0^{s_1}\left|  \frac{1}{|\dot{u}_n^*|^2}\frac{1}{2\pi}\int_0^{2\pi}\hat{\Gamma} d\theta  \right|ds\notag\\
&\leq \frac{C\sqrt{s_1}}{|\alpha_n|}\left(\int_{0}^{ s_1}\int_0^{\pi}|\tau(u_n)|^2 d\theta  ds\right)^{\frac{1}{2}}\notag\\
&\leq \frac{C}{\sqrt{|\alpha_n|}}\|\tau(u_n)\|_{L^2(Q_n)}\leq C\|\tau(u_n)\|_{L^2(Q_n)}\cdot T_n=o(1).
\end{align}
Similarly, one can show that $$\int_0^{s_1}(|\mathbf{II}|+|\mathbf{III}|+|\mathbf{IV}|)ds\leq C$$ and
\begin{equation}\label{inequ:15}
\int_0^{s_1}|\mathbf{VI}|ds=o(1).
\end{equation}

Then we get that $$\|u_n^*(s)\|_{W^{2,1}([0,s_1])}\leq C.$$ Passing to a subsequence, we assume that $$u_n^*(s)\rightharpoonup u^*(s),\ \ \ weakly\ \ in\ \ W^{2,1}([0,s_1]).$$

\

Next, we show the equation for $u^*(s)$.

\

 Noting that, for any $\xi\in [0,s_1]$, there holds
\begin{align*}
&\int_0^{\xi} \frac{1}{|\dot{u}_n^*|^2}\frac{1}{2\pi}\int_0^{2\pi}A(\hat{u}_n)(\nabla \hat{u}_n,\nabla\hat{u}_n) d\theta  ds\\
&= \int_0^{\xi} \frac{1}{|\dot{u}_n^*|^2}\frac{1}{2\pi}\int_0^{2\pi}(A(\hat{u}_n)-A(u_n^*))(\nabla \hat{u}_n,\nabla\hat{u}_n) d\theta  ds+ \int_0^{\xi} \frac{1}{|\dot{u}_n^*|^2}\frac{1}{2\pi}\int_0^{2\pi}A(u_n^*)(\nabla (\hat{u}_n-u_n^*),\nabla\hat{u}_n) d\theta  ds\\
&+\int_0^{\xi} \frac{1}{|\dot{u}_n^*|^2}\frac{1}{2\pi}\int_0^{2\pi}A(u_n^*)(\nabla u_n^*,\nabla(\hat{u}_n-u_n^*)) d\theta  ds
+ \int_0^{\xi} \frac{1}{|\dot{u}_n^*|^2}A(u_n^*)(\nabla u_n^*,\nabla u^*_n)  ds\\&=\mathbf{I_1}+\mathbf{I_2}+\mathbf{I_3}+\mathbf{I_4}.
\end{align*} Here, we in fact used a kind of extension of second fundamental form $A$. Roughly speaking, since $A$ is a smooth function defined on compact Riemannian manifold $N$, then we extend its definition to $N_{\delta}$, a small neighborhood of $N$ with uniform $C^1$-norm. Next, we will do a similar argument to $D\sigma$ and $D^2\sigma$ if needed.

By \eqref{inequ:06}, Lemma \ref{lem:small-energy-regularity}, Lemma \ref{lem:03} and Hölder's inequality, we get
\begin{align*}
\left|\mathbf{I_1}\right|
&\leq \frac{C}{\sqrt{|\alpha_n|}}\int_{- T_n+T_n^{\frac{1}{4}}}^{ T_n-T_n^{\frac{1}{4}}}\int_0^{\pi}|\nabla \hat{u}_n|^2 |\hat{u}_n-u_n^*|d\theta  dt
\\
&\leq \frac{C}{\sqrt{|\alpha_n|}}\sum_{i=0}^{k_n}\|\nabla u_n\|^2_{L^4([- T_n+T_n^{\frac{1}{4}}+i,- T_n+T_n^{\frac{1}{4}}+i+1]\times (0,\pi))} \|\hat{u}_n-u_n^*\|_{L^2([- T_n+T_n^{\frac{1}{4}}+i,- T_n+T_n^{\frac{1}{4}}+i+1]\times (0,\pi))}\\
&\leq Ck_n|\alpha_n| \leq C\sqrt{|\alpha_n|}=o(1)
\end{align*} and
\begin{align*}
\left|\mathbf{I_2}\right|
&\leq \frac{C}{\sqrt{|\alpha_n|}}\int_{- T_n+T_n^{\frac{1}{4}}}^{ T_n-T_n^{\frac{1}{4}}}\int_0^{\pi}|\nabla \hat{u}_n| |\nabla(\hat{u}_n-u_n^*)|d\theta  dt
\\
&\leq \frac{C}{\sqrt{|\alpha_n|}}\sum_{i=0}^{k_n}\|\nabla u_n\|_{L^2([- T_n+T_n^{\frac{1}{4}}+i,- T_n+T_n^{\frac{1}{4}}+i+1]\times (0,\pi))} \|\nabla(\hat{u}_n-u_n^*)\|_{L^2([- T_n+T_n^{\frac{1}{4}}+i,- T_n+T_n^{\frac{1}{4}}+i+1]\times (0,\pi))}\\
&\leq Ck_n\sqrt{|\alpha_n|}o(1)=o(1).
\end{align*}
Similarly, we have $\mathbf{I_3}=o(1)$.

For $\mathbf{I_4}$, by dominated convergence theory, it is easy to see that
\begin{align*}
\mathbf{I_4}=\int_0^{\xi} A(u_n^*)(\frac{d}{ds} u_n^*,\frac{d}{ds} u^*_n) ds= \int_0^{\xi} A(u^*)(\frac{d}{ds} u^*,\frac{d}{ds} u^*) ds+o(1),
\end{align*} which implies
\begin{equation}
\int_0^{\xi}\mathbf{I}ds=-\int_0^{\xi} A(u^*)(\frac{d}{ds} u^*,\frac{d}{ds} u^*) ds+o(1).
\end{equation}

Similarly, we can prove
\begin{equation}
\int_0^{\xi}\mathbf{II}ds=\frac{1}{2}\int_0^{\xi} D^2\sigma(u^*)(\frac{d}{ds} u^*,\frac{d}{ds} u^*) ds+o(1)
\end{equation}
and
\begin{align*}
\int_0^\xi |\mathbf{III}| ds&=\int_0^\xi\left|\frac{\dot{u}_n^*}{|\dot{u}_n^*|^4}\frac{1}{2\pi}\int_0^{2\pi}\langle A(\hat{u}_n)(\nabla \hat{u}_n,\nabla\hat{u}_n),\dot{u}_n^* \rangle d\theta \right| ds\\
&=\int_0^\xi\left|\frac{\dot{u}_n^*}{|\dot{u}_n^*|^4}\frac{1}{2\pi}\int_0^{2\pi}\langle A(u^*_n)(\nabla u^*_n,\nabla u^*_n),\dot{u}_n^* \rangle d\theta \right| ds+o(1)\\
&=\int_0^\xi\left|\langle A(u^*_n)(\frac{d}{ds} u^*_n, \frac{d}{ds} u^*_n),\frac{d}{ds}u_n^* \rangle  \right| ds+o(1)\\
&=\int_0^\xi\left|\langle A(u^*)(\frac{d}{ds} u^*, \frac{d}{ds} u^*),\frac{d}{ds}u^* \rangle  \right| ds+o(1)=o(1)
\end{align*} and
\begin{align*}
\int_0^\xi |\mathbf{IV}| ds=\frac{1}{2}\int_0^\xi\left|\langle D^2\sigma(u^*)(\frac{d}{ds} u^*,\frac{d}{ds} u^*),\frac{d}{ds}u^* \rangle d\theta \right| ds+o(1)=o(1)
\end{align*}
where we used the fact that $$\langle A(u^*)(\frac{d}{ds} u^*, \frac{d}{ds} u^*),\frac{d}{ds}u^* \rangle=\langle D^2\sigma|_{u^*}( \frac{d}{ds}u^*,  \frac{d}{ds}u^*), \frac{d}{ds}u^*\rangle=0.$$

For $\mathbf{V}$ and $\mathbf{VI}$, by \eqref{inequ:14} and \eqref{inequ:15}, we have
\begin{align*}
\int_0^\xi |\mathbf{V}| ds+\int_0^\xi |\mathbf{VI}| ds=o(1).
\end{align*}

Combining these together, we obtain that
\begin{align*}
\frac{d}{ds}u^*(\xi)-\frac{d}{ds}u^*(0)=-\int_0^{\xi} A(u^*)(\frac{d}{ds} u^*,\frac{d}{ds} u^*) ds+ \frac{1}{2}\int_0^{\xi} D^2\sigma(u^*)(\frac{d}{ds} u^*,\frac{d}{ds} u^*) ds,
\end{align*} which implies that $$\frac{d^2}{ds^2}u^*(s)=-A(u^*)(\frac{d}{ds} u^*,\frac{d}{ds} u^*) +\frac{1}{2}D^2\sigma(u^*)(\frac{d}{ds} u^*,\frac{d}{ds} u^*).$$

\

Finally, we compute the length of this geodesic-like curve.

\

Noting that $$\lim_{n\to\infty} \left(|\alpha_n|+\|\tau(u_n)\cdot \partial_tu_n\|_{L^1(Q_n)}+ \|\tau(u_n)\|^2_{L^2(Q_n)}\right)^{\frac{1}{2}}\cdot T_n^{\frac{1}{4}}=0,$$ by the proof of Proposition \ref{prop:03}, we have
\begin{align*}
osc_{[- T_n+1, -T_n+T_n^{\frac{1}{4}}]\times [0,\pi]}u_n+osc_{[ T_n-T_n^{\frac{1}{4}},T_n-1]\times [0,\pi]}u_n=o(1).
\end{align*}
Combining this with Lemma \ref{lem:03}, the length of the limit geodesic-like curve $\gamma$ is
\begin{align*}
L(\gamma)&=\lim_{n\to\infty}\int_{-T_n+T_n^{\frac{1}{4}}}^{T_n-T_n^{\frac{1}{4}}}|\dot{u}_n^*|dt\\
&=\lim_{n\to\infty}\sum_{i=1}^{k_n}\int_{-T_n+T_n^{\frac{1}{4}}+(i-1)}^{-T_n+T_n^{\frac{1}{4}}+i}|\dot{u}_n^*|dt\\
&=\lim_{n\to\infty}k_n\sqrt{|\alpha_n|}\left(\frac{1}{\sqrt{\pi}}+o(1)\right)\\
&=\lim_{n\to\infty}2 T_n\sqrt{|\alpha_n|}\left(\frac{1}{\sqrt{\pi}}+o(1)\right)=\frac{2}{\sqrt{\pi}}\mu.
\end{align*}

Lastly, if $\mu=+\infty$, from the above argument, it is easy to see that the neck contains at least an infinite length geodesic-like curve $\gamma$.
\end{proof}

\

At the end of this section, we give the proof of Theorem \ref{thm-main:01} and Theorem \ref{thm-main:03}.
\begin{proof}[\textbf{Proof of Theorem \ref{thm-main:01} and Theorem \ref{thm-main:03}.}]
We just prove conclusion $(2)$ of Theorem \ref{thm-main:03}, since the proof of Theorem \ref{thm-main:01} and conclusion $(1)$ of Theorem \ref{thm-main:03} is similar and in fact easier. In this  case,  cutting $(S^1\times [0,\pi],g_n)$ along the geodesic $\gamma_n$ and then we can see it as  $Q_n:=([-\frac{1}{2}\rho_n^{-1},\frac{1}{2}\rho_n^{-1}]\times \rho_n\frac{1}{\pi}[0,\pi],ds^2+d\theta^2)$. For a sequence of maps $u_n:Q_n\to (N,h)$ with tension fields $\tau(u_n)$ and with free boundary condition $u_n|_{\partial Q_n}\subset K$, where the boundary $\partial Q_n=[-\frac{1}{2}\rho_n^{-1},\frac{1}{2}\rho_n^{-1}]\times \{0,\pi\}$, we define $\tilde{Q}_n=[-\frac{1}{2}\pi\rho_n^{-2},\frac{1}{2}\pi\rho_n^{-2}]\times [0,\pi] $ and $v_n(x)=u_n(\frac{1}{\pi}\rho_nx)$, $x\in \tilde{Q}_n$. Then it is easy to see that $v_n(x)$ is a map from $\tilde{Q}_n$ to $N$ with free boundary $v_n(x)|_{\partial \tilde{Q}_n}\subset K$ and with tension field $\tau(v_n)(x)=(\frac{\rho_n}{\pi})^2\tau(u_n)(\frac{1}{\pi}\rho_nx)$. Moreover, we have
\begin{align*}
E(v_n(x),\tilde{Q}_n)=E(u_n(x),Q_n),\ \ \left| \int_{\tilde{Q}_n}\phi(v_n)dx  \right|=\left| \int_{Q_n}\phi(u_n)dx  \right|,\\
\|\tau(v_n)\cdot \frac{\partial v_n}{\partial t}\|_{L^1(\tilde{Q}_n)}=\frac{\rho_n}{\pi}\|\tau(u_n)\cdot \frac{\partial u_n}{\partial t}\|_{L^1(Q_n)},\ \ \|\tau(v_n)\|_{L^2(\tilde{Q}_n)}=\frac{\rho_n}{\pi}\|\tau(u_n)\|_{L^2(Q_n)}.
\end{align*}

Since $v_n(x)$ may have blow-up points on $\tilde{Q}_n$,  by  a standard decomposition argument as in \cite{Zhu-3}, we can divide $\tilde{Q}_n$ as  $$\tilde{Q}_n=I_n^0\cup J_n^1\cup I_n^1\cup J_n^2....\cup J_n^L\cup I_n^L,$$
where $L$ is a uniformly bounded positive integer which is independent of $n$. Moreover, there is no blow-up point on each $I_n^i,\ i=0,1,...,L,$ i.e. \eqref{equ:03} holds. Using Theorem \ref{thm:01} and Theorem \ref{thm:02} on each $I_n^i,\ i=0,1,...,L,$ and classical blow-up theory \cite{DingWeiyueandTiangang,jost-Liu-Zhu} on each $J_n^i,\ i=1,...,L$, we can obtain the conclusion $(2)$ of Theorem \ref{thm-main:03}.
\end{proof}

\

\begin{proof}[\textbf{Proof of Theorem \ref{thm-main:02}}]
Firstly, it is easy to see that conclusion $(1)$ is a direct consequence of Theorem \ref{thm-main:01}. Secondly, if there is no bubble at infinity, then by \cite{Ding-Li-Liu}, we can find a time sequence $t_n\to\infty$ such that  $$ \|\tau(u(x,t_n),g(t_n))\|_{L^2(M)}=o(1)\rho(t_n)^2,\ \ E(u(x,t_n);g(t_n))\leq C\rho(t_n)^2,$$ which immediately implies the condition  \eqref{equ:17}. Then conclusion $(2)$ also follows from Theorem \ref{thm-main:01}.
\end{proof}

\

\section{Evolution of minimal cylinders with free boundary}\label{sec:flow}

\

In this section, we shall extend Theorem \ref{thm:03} and Theorem \ref{thm-main:02} to the free boundary case. We will study the convergence of the flow \eqref{equ:15}-\eqref{equ:16} and prove Theorem \ref{thm-main:04}.

\

Firstly, we derive some basic properties of the flow. Let $M=S^1\times [0,\pi]$.

\begin{lem}\label{lem:01}
Suppose $(u,g)$ is a smooth solution of the flow \eqref{equ:15}-\eqref{equ:16} in $M\times [0,T)$. Then for any $t\in [0,T)$, there holds
\begin{align*}
\frac{d}{dt}\frac{1}{2}\int|\nabla_gu|^2dM=&-\int_M|\tau_g(u)|^2dM-\frac{1}{4}\left|\int_M\phi(u)dM\right|^2,
\end{align*} where $$\phi(u):=\left|\frac{\partial u}{\partial \theta^1}\right|^2-\left|\frac{\partial u}{\partial \theta^2}\right|^2-2\sqrt{-1}\frac{\partial u}{\partial \theta^1}\cdot \frac{\partial u}{\partial \theta^2}$$ and $\phi(u)dz^2$ is the Hopf differential of $u$, $dz=d\theta^1+\sqrt{-1}d\theta^2$. Here, $(\theta^1,\theta^2)$ is a new coordinate system defined by
\begin{align*}
\left(
  \begin{array}{cc}
    \theta^1 \\ \theta^2
  \end{array}
\right)= \frac{1}{\sqrt{b}}\left(
                             \begin{array}{cc}
                               1 & a \\
                               0 & b \\
                             \end{array}
                           \right)
                           \cdot
                           \left(
  \begin{array}{cc}
    x^1 \\ x^2
  \end{array}
\right).
\end{align*}
\end{lem}
\begin{proof}
For ant $t\in [0,T)$, since the volume element $dM$ is independent of time $t$, we have
\begin{align*}
\frac{d}{dt}\frac{1}{2}\int|\nabla_gu|^2dM=\int_M \left( g^{ij}\langle \frac{d}{dt}\frac{\partial u}{\partial x^i},  \frac{\partial u}{\partial x^j}\rangle +\frac{1}{2} \frac{d}{dt}g^{ij}\langle\frac{\partial u}{\partial x^i},  \frac{\partial u}{\partial x^j}\rangle\right) dM.
\end{align*}

A direct computation yields
\begin{align*}
\int_M \left( g^{ij}\langle \frac{d}{dt}\frac{\partial u}{\partial x^i},  \frac{\partial u}{\partial x^j}\rangle\right) dM&= -\int_M\langle \Delta_gu,\frac{\partial u}{\partial t} \rangle dM + \int_{\partial M}\langle\frac{\partial u}{\partial \overrightarrow{n}},\frac{\partial u}{\partial t}\rangle \\ &=-\int_M |\tau_g(u)|^2 dM
\end{align*} where we used the free boundary condition that $\langle\frac{\partial u}{\partial \overrightarrow{n}},\frac{\partial u}{\partial t}\rangle|_{\partial M}=0$.

By \cite{Ding-Li-Liu}, we get
\begin{align*}
&\int_M \left(  \frac{d}{dt}g^{ij}\langle\frac{\partial u}{\partial x^i},  \frac{\partial u}{\partial x^j}\rangle\right)  dM\\&=\frac{2\dot{a}}{b}\int_M\left(a\left|\frac{\partial u}{\partial x^1}\right|^2-\frac{\partial u}{\partial x^1}\cdot \frac{\partial u}{\partial x^2}\right)dM+\frac{\dot{b}}{b^2}\int_M\left((b^2-a^2)\left|\frac{\partial u}{\partial x^1}\right|^2+2a\frac{\partial u}{\partial x^1}\cdot \frac{\partial u}{\partial x^2}-\left|\frac{\partial u}{\partial x^2}\right|^2\right)dM\\
&=-2\left|\int_M\left(\frac{\partial u}{\partial \theta^1}\cdot \frac{\partial u}{\partial \theta^2}\right)dM \right|^2 -\frac{1}{2}\left|\int_M \left(\left|\frac{\partial u}{\partial \theta^1}\right|^2-\left|\frac{\partial u}{\partial \theta^2}\right|^2\right) dM \right|^2=-\frac{1}{2}\left|\int_M\phi(u)dM\right|^2.
\end{align*} Then the conclusion of the lemma follows.
\end{proof}

\

With the help of Lemma \ref{lem:01}, by \cite{Ding-Li-Liu}, we can also show that the metric $g$ along the flow \eqref{equ:15}-\eqref{equ:16} will not go to the boundary of the Teichm\"{u}ller space at any finite time.

\begin{cor}
Suppose $(u,g)$ is a smooth solution of the flow \eqref{equ:15}-\eqref{equ:16} in $M\times [0,T)$. Then for any $t\in [0,T)$, there holds $$C^{-1}g_0\leq g(t)\leq Cg_0,$$ where $C>0$ depends only on $g_0,\ E_0=E(u_0,g_0)$ and $T$.
\end{cor}

\begin{prop}
If the flow exists globally and converges smoothly to $(u_\infty,g_\infty)$ as $t\to \infty$, where $u_\infty$ is not a constant map, then $u_\infty$ must be a branched minimal immersion which is homotopic to the initial surface.
\end{prop}
\begin{proof}
By the same proof as in Proposition 2.4 of \cite{Ding-Li-Liu}, we can find a time sequence $t_n\to\infty$ such that $(u(t_n),g(t_n))\to (u_\infty,g_\infty)$. Moreover, there hold $$\tau_\infty(u_\infty)=0\ \ and\ \ \int_{M}\phi(u_\infty)dx=0,$$ where $\phi(u)dz^2$ is the Hopf differential of $u$.

Since $\phi(u_\infty)$ is holomorphic on $M$ and $Im\phi(u_\infty)|_{\partial M}=0$, we know that $\phi(u_\infty)\equiv C$ where $C$ is a constant. Then we have $C=0$ which implies that $u_\infty$ must be a branched minimal immersion.
\end{proof}

With the help of above lemmas, by apply similar arguments as in \cite{Str1,Str2,Ma,Ding-Li-Liu}, we can show the following long time existence result. Since there are no technical difficulty, we will not give the detailed proof here.
\begin{thm}
For any initial value $g_0\in\mathbf{g}$ and $u_0\in W^{1,2}(M,N)$, there exist a unique solution $(u,g)$ of \eqref{equ:15}-\eqref{equ:16} on $M\times (0,\infty)$ and a finite singular points set $\mathbf{S}=\{(x_i,t_i)\}_{i=1}^m$,  such that $$g(t)\in C^{\frac{\alpha}{2}}_{loc}(M\times [0,\infty)\setminus \mathbf{S}),\ \ u(x,t)\in C^{2+\alpha,1+\frac{\alpha}{2}}_{loc}(M\times (0,\infty)\setminus \mathbf{S}).$$ Moreover, at each blow-up point $(x_i,t_i)$, there are some bubbles $\{w^j\}_{j=1}^{l_i}$ splitting off where each $w^j$ is either a harmonic $S^2$ or a harmonic $D_1(0)$ with free boundary condition $w^j(\partial D_1(0))\subset K$.
\end{thm}

Lastly, in order to study the convergence of the flow, i.e. Theorem \ref{thm-main:04}, we first need to establish the following key proposition.
\begin{prop}\label{prop:02}
Suppose $(u,g)$  is a solution of \eqref{equ:15}-\eqref{equ:16} on $M\times (0,\infty)$. If
$$\lim_{t\to\infty}E(u(t),g(t))=0$$ and
 $$ \limsup_{t\to\infty}\left(\rho(t)^p\|\tau_{g_t}(u)\|_{L^2(M)}+t^r\rho(t)\right)\leq C$$ for some $p\in [0,1)$, $r\in (0,\frac{1}{100})$, then there exists a time sequence $t_n\to\infty$, such that  $$ \|\tau(u(x,t_n),g(t_n))\|_{L^2(M)}=o(1)\rho(t_n)^2,\ \ E(u(x,t_n);g(t_n))\leq C\rho(t_n)^2.$$
\end{prop}

We divide the proof of Proposition \ref{prop:02} into two steps corresponding to two types of degenerations.

\

\noindent\textbf{Type I degeneration:} the degenerated curve is a closed geodesic curve.

\

Let $\rho(t)$ be the length of the shortest closed geodesic of $(M,g(t))$. In this case, we can see $(M,g(t))$ as a cylinder $P_t=[-\frac{1}{2}\rho(t)^{-1},\frac{1}{2}\rho(t)^{-1}]\times \rho(t)\frac{1}{2\pi}S^1$ and view $u(t)$ as a map from $P_t$ to $N$ with free boundary condition $u(t)|_{\partial P_t}\subset K$.

Set $\tilde{P}_t=2\pi \rho(t)^{-1}P_t=[-\pi\rho(t)^{-2},\pi\rho(t)^{-2}]\times S^1$ and $v(t,x)=u(t,\frac{1}{2\pi}\rho(t)x)$, $x\in \tilde{P}_t$. Then it is easy to see that $v(t,x)$ is a map from $\tilde{P}_t$ to $N$ with free boundary $v(t,x)|_{\partial \tilde{P}_t}\subset K$ and with tension field $\tau(v)(t,x)=(\frac{\rho(t)}{2\pi})^2\tau(u)(t,\frac{1}{2\pi}\rho(t)x)$. Moreover, we have
\begin{align*}
E(v(t,x),\tilde{P}_t)=E(u(t,x),P_t),\ \ \left| \int_{\tilde{P}_t}\phi(v)dx  \right|=\left| \int_{P_t}\phi(u)dx  \right|,\\
\|\tau(v)\|_{L^1(\tilde{P}_t)}=\|\tau(u)\|_{L^1(P_t)},\ \ \|\tau(v)\|_{L^2(\tilde{P}_t)}=\frac{\rho(t)}{2\pi}\|\tau(u)\|_{L^2(P_t)}.
\end{align*}

\

Denote $T_t:=\pi\rho(t)^{-2}$, then $\tilde{P}_t= [-T_t,T_t]\times S^1$. Since $\sup_t\rho(t)^p\|\tau_{g_t}(u)\|_{L^2(M)}\leq C,\ p\in [0,1)$, we have $$\|\tau(v)\|_{L^2(\tilde{P}_t)}\leq C T_t^{-\frac{1-p}{2}}.$$  Thanks to the free boundary condition, we first have  following energy estimate near the boundary.

\begin{lem}\label{lem:05}
Under the assumptions of Proposition \ref{prop:02}, we have
$$E(v;[-T_t,-T_t+1]\times S^1)+E(v;[T_t-1,T_t]\times S^1)\leq C\left(\rho^2(t)E(v)+\sqrt{E(v)}\|\tau(v)\|_{L^2}+ \|\tau(v)\|^2_{L^2}\right).$$
\end{lem}

\begin{proof}

By the definition of Pohozaev type constant $\alpha(t)$ for $v$, we have
\begin{align*}
|\alpha_tT_t|&=\left|\int_{[0,T_t]\times S^1}\left(\left|\frac{\partial v}{\partial s}\right|^2-\left|\frac{\partial v}{\partial \theta}\right|^2\right)dsd\theta-\int_0^{T_t}\int_{[0,\xi]\times S^1}\tau(v)\cdot \frac{\partial v}{\partial s}dsd\theta d\xi\right|\\
&\leq 2E(v)+ \|\tau(v)\|_{L^2}\sqrt{E(v)}T_t
\end{align*} which implies that
\begin{equation}
|\alpha_t|\leq C\left(\frac{1}{T_t}E(v)+\sqrt{E(v)}\|\tau(v)\|_{L^2(\tilde{P}_t)}\right).
\end{equation}
Noting that $\|\tau(v)\|_{L^2(\tilde{P}_t)}\leq CT_t^{-\frac{1-p}{2}}$, then for any $q\in (0,\frac{1-p}{2})$, we have $$\lim_{t\to\infty}\left(\|\tau(v)\|^2_{L^2(\tilde{P}_t)}+\|\tau(v)\cdot \partial_t v\|_{L^1(\tilde{P}_t)}\right)\cdot T_t^{q}=0,\ \ and\ \ \lim_{t\to\infty}\alpha_t\cdot T_t^{q}=0.  $$

Since there is no energy concentration, i.e. $\lim_{t\to\infty}E(v;\tilde{P}_t)=0$, by the proof of Proposition \ref{prop:03}, we have $$osc_{[-T_t,-T_t+T_t^{\frac{q}{2}}]\times S^1}v=o(1).$$  Thus we can extend $v(t,x)$ across the free boundary $\{-T_t\}\times S^1$. Set $w(s,\theta):=v(-T_t+s,\theta),\ (s,\theta)\in [0,T_t^{\frac{q}{2}}]\times S^1$. Then $osc_{[0,T_t^{\frac{q}{2}}]\times S^1}w=o(1)$. Since $w$ satisfies the free boundary condition $w|_{\{0\}\times S^1}\subset K$, by involution map $\sigma$ (see \eqref{def:01}), similarly to definition \eqref{def:function}, we define
\begin{align}\label{def:function-2}
\hat{w}(s,\theta)=
\begin{cases}
w(s,\theta),\quad &if \quad (s,\theta)\in [0,T_t^{\frac{q}{2}}]\times S^1;\\
\sigma(w(-s,\theta)) ,\quad &if \quad (t,\theta)\in [-T_t^{\frac{q}{2}},0]\times S^1.
\end{cases}
\end{align}
By Proposition \ref{prop:01}, we know that $\hat{w}\in W^{2,p}([-T_t^{\frac{q}{2}},T_t^{\frac{q}{2}}]\times S^1)$ and
\begin{equation}
\Delta \hat{w}+\Upsilon_{\hat{w}}(\nabla\hat{w},\nabla\hat{w})=\hat{\Gamma}\quad in \quad [-T_t^{\frac{q}{2}},T_t^{\frac{q}{2}}]\times S^1,
\end{equation}
where $\Upsilon_{\hat{w}}(\cdot,\cdot)$ is a bounded bilinear form defined by
\begin{align*}
\Upsilon_{\hat{w}}(\cdot,\cdot)=
\begin{cases}
A(\hat{w})(\cdot,\cdot)\ &in\ [0,T_t^{\frac{q}{2}}]\times S^1,\\
A(\hat{w})(\cdot,\cdot)-D^2\sigma|_{\sigma(\hat{w})}(D\sigma|_{\hat{w}}\circ\cdot\ ,
D\sigma|_{\hat{w}}\circ\cdot)\ &in\ [-T_t^{\frac{q}{2}},0]\times S^1;
\end{cases}\end{align*} satisfying
\[
|\Upsilon_{\hat{w}}(\nabla\hat{w},\nabla\hat{w})|\leq C(K,N)|\nabla\hat{w}|^2
\] and $\hat{\Gamma}$ is defined by
\begin{align*}
\hat{\Gamma}=
\begin{cases}
\tau(w)(x)\ &in\ [0,T_t^{\frac{q}{2}}]\times S^1,\\
D\sigma|_{\sigma(\hat{w})}\circ \tau(w)(\rho(x))\ &in\ [-T_t^{\frac{q}{2}},0]\times S^1.
\end{cases}\end{align*}

Similarly to the proof of \eqref{inequ:01}, \eqref{inequ:02}, \eqref{inequ:03} and \eqref{inequ:10} (taking $t_n=0$), we get  that for any $\xi\in [0,T_t^{\frac{q}{2}}]$, there holds
\begin{align*}
&\int_{0}^{\xi}\int_{S^1}|\nabla w|^2d\theta ds+\int_{0}^{\xi}\int_{S^1}\left(\left|\frac{\partial w}{\partial\theta}\right|^2- \left|\frac{\partial w}{\partial s}\right|^2\right)d\theta ds\notag\\
&\leq C\epsilon\int_{0}^{\xi}\int_{S^1}|\nabla w|^2dtd\theta+\frac{1}{8}\int_{0}^{\xi}\int_{S^1}|\nabla w|^2dsd\theta +C\|\tau(w)\|^2_{L^2([0,\sqrt{T}_t]\times S^1)}+3\int_{\{\xi\}\times S^1}|\nabla w|^2d\theta.
\end{align*} Denoting $f(\xi):=\int_{0}^{\xi}\int_{S^1}|\nabla w|^2d\theta ds$ and taking $C\epsilon\leq\frac{1}{8}$, we get
\begin{align*}
\frac{3}{4}f(\xi)&\leq  3f'(\xi) +C\|\tau(w)\|^2_{L^2(Q_n^+)}+\int_{0}^{\xi}\int_{S^1}\left( \left|\frac{\partial w}{\partial s}\right|^2-\left|\frac{\partial w}{\partial\theta}\right|^2\right)dsd\theta\\&\leq 3f'(\xi)+C\|\tau(w)\|^2_{L^2(Q_n^+)}+|\alpha_n|\xi +2\int_{0}^{\xi}\int_0^{-T_t+s}\int_{S^1}\tau(v)\cdot \partial_t v dtd\theta ds\\&\leq 3f'(\xi)+|\alpha_n|\xi +C\|\tau(v)\|^2_{L^2(Q_n^+)}+2\sqrt{E}\|\tau(v)\|_{L^2}\xi\\
&\leq 3f'(\xi)+|\alpha_n|\xi +C\left(\|\tau(v)\|^2_{L^2}+\sqrt{E}\|\tau(v)\|_{L^2}\right)(1+\xi)
\end{align*} which implies
\begin{equation}
\left(e^{-\frac{1}{4}\xi}f(\xi)\right)'\geq -|\alpha_t|\xi e^{-\frac{1}{4}\xi}-C\left(\|\tau(v)\|^2_{L^2}+\sqrt{E}\|\tau(v)\|_{L^2}\right) (1+\xi)e^{-\frac{1}{4}\xi}.
\end{equation}

Integrating from $2$ to $\frac{1}{2}T_t^{\frac{q}{2}}$, we get
\begin{align*}
f(2)&=\int_{0}^2\int_0^{2\pi}|\nabla w|^2d\theta dt\\&\leq Ce^{-\frac{1}{8}T_t^{\frac{q}{2}}}f(\frac{1}{2}T_t^{\frac{q}{2}})+ C|\alpha_t|+C\left(\|\tau(v)\|^2_{L^2}+\sqrt{E}\|\tau(v)\|_{L^2}\right)\\
&\leq Ce^{-\frac{1}{8}T_t^{\frac{q}{2}}}\int_{-T_t}^{-T_t+\frac{1}{2}T_t^{\frac{q}{2}}}\int_0^{2\pi}|\nabla v|^2d\theta dt + C|\alpha_t|+C\left(\|\tau(v)\|^2_{L^2}+\sqrt{E}\|\tau(v)\|_{L^2}\right).
\end{align*}

Thus, we get
\begin{align*}
\int_{[-T_t,-T_t+2]\times S^1}|\nabla v|^2dsd\theta&\leq  C\left(\frac{1}{T_t}E(v)+\sqrt{E(v)}\|\tau(v)\|_{L^2}+ \|\tau(v)\|^2_{L^2}\right)\\
&\leq C\left(\rho^2(t)E(v)+\sqrt{E(v)}\|\tau(v)\|_{L^2}+ \|\tau(v)\|^2_{L^2}\right).
\end{align*}
Similarly, we can prove
\begin{equation*}
\int_{[T_t-2,T_t]\times S^1}|\nabla v|^2dsd\theta\leq C\left(\rho^2(t)E(v)+\sqrt{E(v)}\|\tau(v)\|_{L^2}+ \|\tau(v)\|^2_{L^2}\right).
\end{equation*}
We proved the lemma.
\end{proof}

\begin{lem}\label{lem:06}
Under the assumptions of Proposition \ref{prop:02}, we have
\begin{align*}
\frac{d}{dt}E=-\|\tau(u)\|^2_{L^2}-\frac{1}{4}|\Phi|^2=-(1-C\rho^2(t))E^2-(1-o(1))\|\tau(u)\|^2_{L^2}+CE^{\frac{5}{2}},
\end{align*} where $|\Phi|:=\left|\int_M\phi(u)dM\right|$. Moreover, this yields $$E(t)\leq C\rho(t)^2 $$ when $t$ is big enough.
\end{lem}

\begin{proof}
Denote $v^*(s)=\frac{1}{2\pi}\int_{\{s\}\times S^1}vd\theta$. By small energy regularity, we have $$\|v-v^*\|_{L^\infty(\tilde{P}_t)}\leq C(\sqrt{E(v)}+\|\tau(v)\|_{L^2}).$$ Integrating by parts, by Lemma \ref{lem:05}, we get
\begin{align*}
\int_{\tilde{P}_t}\left|\frac{\partial v}{\partial \theta} \right|^2dsd\theta&=\int_{\tilde{P}_t}\nabla v\nabla (v-v^*)dsd\theta\\&=-\int_{\tilde{P}_t}\Delta v\cdot (v-v^*)dsd\theta+\int_{\partial \tilde{P}_t}\frac{\partial v}{\partial s} (v-v^*)d\theta\\
&\leq C(\sqrt{E(v)}+\|\tau(v)\|_{L^2})(E(v)+\|\tau(v)\|_{L^2})+C\int_{\partial \tilde{P}_t}|\nabla v|^2d\theta\\
&\leq C(\sqrt{E(v)}+\|\tau(v)\|_{L^2})(E(v)+\|\tau(v)\|_{L^2})\\&\quad +CE(v;[-T_t,-T_t+1]\times S^1)+E(v;[T_t-1,T_t]\times S^1)\\
&=C\left(\rho^2(t)E(v)+E^{\frac{3}{2}}(v)+\sqrt{E(v)}\|\tau(v)\|_{L^2}+ \|\tau(v)\|^2_{L^2}\right).
\end{align*}

Since
\begin{align*}
|\Phi|^2&\geq 4E^2(v)-8E\int_{\tilde{P}_t}\left|\frac{\partial v}{\partial \theta} \right|^2dsd\theta\\&=(4-C\rho^2(t))E^2(v)-CE^{\frac{5}{2}}-o(1)\|\tau(v)\|^2_{L^2},
\end{align*} then
\begin{align}\label{equ:07}
\frac{d}{dt}E=-\|\tau(u)\|^2_{L^2}-\frac{1}{4}|\Phi|^2=-(1-C\rho^2(t))E^2-(1-o(1))\|\tau(u)\|^2_{L^2}+CE^{\frac{5}{2}}.
\end{align} Therefore when $t$ is big enough, we have the following differential inequality
$$\frac{d}{dt}E\leq -\frac{1}{2}E^2,$$
which immediately implies  the following decay that $E(t)\leq \frac{4}{t}$ (when $t$ is big enough).

Since $\rho(t)\leq Ct^{-r}$ where $r\in (0,\frac{1}{100})$, using \eqref{equ:07} again, we get
\begin{align*}
\frac{d}{dt}E\leq -E^2+Ct^{-2r}E^2+CE^{\frac{5}{2}}
\end{align*} which implies $$\frac{d}{dt}E^{-1}\geq 1-Ct^{-2r}.$$ Then we get $$E(t)\leq \frac{1}{t}+\frac{C}{t^{1+2r}},$$ when $t$ is big enough.

From the equation that $$\frac{d}{dt}b=-\frac{b}{2}\int_M \left(\left|\frac{\partial u}{\partial \theta^1}\right|^2-\left|\frac{\partial u}{\partial \theta^2}\right|^2\right) d\theta^1d\theta^2,$$ where $(\theta^1,\theta^2)$ is the coordinate system defined in Lemma \ref{lem:01}, we have
\begin{equation}\label{inequ:12}
|b^{-1}\frac{d}{dt}b|\leq E \leq \frac{1}{t}+\frac{C}{t^{1+2r}}.\end{equation} Then $$|\log b(t)|\leq \log t+C,$$
which is
\begin{equation}\label{inequ:13}
b(t)+\frac{1}{b(t)}\leq Ct.
 \end{equation}
This implies $\rho(t)=\frac{1}{\sqrt{b(t)}}\geq \frac{C}{\sqrt{t}}$ and $$E(t)\leq \frac{C}{t}\leq C\rho(t)^2.$$
\end{proof}

\

\begin{proof}[\textbf{Proof of Proposition \ref{prop:02} for type I degeneration}]

We first prove the conclusion that there exists a time sequence $t_n\to\infty$ such that $$\|\tau(u(t_n),g_n)\|_{L^2(M)}=o(1)\rho(t_n)^2.$$  If not, then we suppose that for all $t>1$, there exists a positive constant  $\delta>0$ such that $$\|\tau(u)\|_{L^2}\geq\delta\rho^{2}(t).$$
Now, noting the fact that $$\rho(t)=\frac{1}{\sqrt{b(t)}}\geq \frac{C}{\sqrt{t}},\ \ \|\tau(u)\|_{L^2}\geq \delta\rho^2(t)\geq Ct^{-1},$$ by \eqref{equ:07}, when $t$ is big enough, we get that
\begin{align*}
\frac{d}{dt}E
\leq -E^2-Ct^{-2}\leq -(1+\frac{C}{2})E^2,
\end{align*} which implies $$E(t)\leq (1-\epsilon_0)t^{-1},$$ where $\epsilon_0=\frac{C}{4}$ is a small positive constant.
Using \eqref{inequ:12} again, we can improve estimate of $\rho(t)$ to the following that  $$\rho(t)\geq Ct^{\frac{1}{2}\epsilon_0-\frac{1}{2}}.$$

Since $E(t)\leq Ct^{-1}$ and $\lim_{t\to\infty}E(t)=0$, then we can choose a time sequence $t_n\to\infty$ such that $$ -\frac{d}{dt}E(t_n)\leq Ct_n^{-2}.$$ By Lemma \ref{lem:06}, we have $$\|\tau(u(t_n))\|_{L^2}\leq Ct_n^{-1}\leq Ct_n^{-\epsilon_0}\rho^2(t_n)=o(1)\rho^2(t_n),$$ which is a contradiction. Thus, we obtain that there exists a time sequence $t_n\to\infty$ such that $$\|\tau(u(t_n))\|_{L^2}=o(1)\rho^2(t_n).$$ By Lemma \ref{lem:06}, we get $E(u(t_n))\leq C\rho^2(t_n)$.
\end{proof}

\

\noindent\textbf{Type II degeneration:} the degenerated curve is a geodesic which connects two components of the boundary.

\

\begin{proof}[\textbf{Proof of Proposition \ref{prop:02} for type II degeneration}]
Let $\rho(t)$ be the length of the shortest geodesic of $(M,g(t))$ which connects two components of the boundary $\partial M$. In this case, cutting $(M,g(t))$ along its shortest geodesic we get a half-cylinder $Q_t=[-\frac{1}{2}\rho(t)^{-1},\frac{1}{2}\rho(t)^{-1}]\times \rho(t)\frac{1}{\pi}[0,\pi]$. Then we can view $u(t)$ as a map from $Q_t$ to $N$ with free boundary condition $u(t)|_{\partial Q_t}\subset K$, where $\partial Q_t=[-\frac{1}{2}\rho(t)^{-1},\frac{1}{2}\rho(t)^{-1}]\times \{0\} \cup [-\frac{1}{2}\rho(t)^{-1},\frac{1}{2}\rho(t)^{-1}]\times \{\pi\}$.

Set $\tilde{Q}_t=\pi \rho(t)^{-1}Q_t=[-\frac{1}{2}\pi\rho(t)^{-2},\frac{1}{2}\pi\rho(t)^{-2}]\times [0,\pi] $ and $v(t,x)=u(t,\frac{1}{\pi}\rho(t)x)$, $x\in \tilde{Q}_t$. Then it is easy to see that $v(t,x)$ is a map from $\tilde{Q}_t$ to $N$ with free boundary $v(t,x)|_{\partial \tilde{Q}_t}\subset K$ and with tension field $\tau(v)(t,x)=(\frac{\rho(t)}{\pi})^2\tau(u)(t,\frac{1}{\pi}\rho(t)x)$. Moreover, we have
\begin{align*}
E(v(t,x),\tilde{Q}_t)=E(u(t,x),Q_t),\ \ \left| \int_{\tilde{Q}_t}\phi(v)dx  \right|=\left| \int_{Q_t}\phi(u)dx  \right|,\\
\|\tau(v)\|_{L^1(\tilde{Q}_t)}=\|\tau(u)\|_{L^1(Q_t)},\ \ \|\tau(v)\|_{L^2(\tilde{Q}_t)}=\frac{\rho(t)}{\pi}\|\tau(u)\|_{L^2(Q_t)}.
\end{align*}

We still denote $T_t=\frac{1}{2}\pi\rho(t)^{-2}$, then $\tilde{Q}_t=[-T_t,T_t]\times [0,\pi]$. Since there is no energy concentration and the fact that $\|\tau(v)\|_{L^2(Q_t)}\leq C\rho(t)^{1-p}=o(1)$, by Lemma \ref{lem:small-energy-regularity} and Lemma \ref{lem:small-energy-regularity-1}, we have  $$osc_{\tilde{Q}_t}v\leq C(E(v)+\|\tau(v)\|_{L^2})=o(1).$$ Then by \eqref{def:function}, we can extend $v$ across the boundary to $\hat{v}:\hat{Q}_t=[-T_t,T_t]\times S^1\to N$. We want to remark that  $\hat{v}|_{\{-T_t\}\times S^1}\equiv \hat{v}|_{\{T_t\}\times S^1}$. Integrating by parts, by Proposition \ref{prop:01} , we get
\begin{align*}
\int_{\hat{Q}_t}\left|\frac{\partial \hat{v}}{\partial \theta} \right|^2dsd\theta&=\int_{\hat{Q}_t}\nabla \hat{v}\nabla (\hat{v}-\hat{v}^*)dsd\theta\\&=-\int_{\hat{Q}_t}\Delta \hat{v}\cdot (\hat{v}-\hat{v}^*)dsd\theta\\
&\leq C(\sqrt{E(v)}+\|\tau(v)\|_{L^2})(E(v)+\|\tau(v)\|_{L^2})\\
&\leq C\left(E^{\frac{3}{2}}(v)+\sqrt{E(v)}\|\tau(v)\|_{L^2}+ \|\tau(v)\|^2_{L^2}\right).
\end{align*}

Since
\begin{align*}
|\Phi|^2&\geq 4E^2(v)-8E\int_{\tilde{Q}_t}\left|\frac{\partial v}{\partial \theta} \right|^2dsd\theta\\&=4E^2(v)-CE^{\frac{5}{2}}-o(1)\|\tau(v)\|^2_{L^2},
\end{align*} then
\begin{align*}
\frac{d}{dt}E=-\|\tau(u)\|^2_{L^2}-\frac{1}{4}|\Phi|^2=-E^2-(1-o(1))\|\tau(u)\|^2_{L^2}+CE^{\frac{5}{2}}.
\end{align*} Therefore when $t$ is big enough, we have the following differential inequality
$$\frac{d}{dt}E\leq -\frac{1}{2}E^2,$$
which immediately implies  the following decay that $E(t)\leq \frac{4}{t}$ (when $t$ is big enough). Then
$$\frac{d}{dt}E^{-1}\geq 1-Ct^{-\frac{1}{2}},$$ which implies  $$E(t)\leq \frac{1}{t}+\frac{C}{t^{\frac{3}{2}}},$$ when $t$ is big enough. With above estimate, the rest proof is the same as the case of type I degeneration.
\end{proof}

\begin{proof}[\textbf{Proof of Theorem \ref{thm-main:04}}]
By Proposition \ref{prop:02}, the conclusions of Theorem \ref{thm-main:04} is a consequence of Theorem \ref{thm-main:03}.
\end{proof}

\


\providecommand{\bysame}{\leavevmode\hbox to3em{\hrulefill}\thinspace}
\providecommand{\MR}{\relax\ifhmode\unskip\space\fi MR }
\providecommand{\MRhref}[2]{%
  \href{http://www.ams.org/mathscinet-getitem?mr=#1}{#2}
}
\providecommand{\href}[2]{#2}

\end{document}